\documentclass[12pt,a4paper]{amsart}
\usepackage[czech,english]{babel}
\usepackage[T1]{fontenc}
\usepackage{amssymb}
\usepackage[all,cmtip]{xy}

\calclayout

\newcommand{\+}{\protect\nobreakdash-}
\newcommand{\<}{\protect\nobreakdash--}
\renewcommand{\:}{\colon}

\newcommand{\rarrow}{\longrightarrow}

\newcommand{\ot}{\otimes}

\newcommand{\lrarrow}{\mskip.5\thinmuskip\relbar\joinrel\relbar\joinrel
 \rightarrow\mskip.5\thinmuskip\relax}

\newcommand{\bu}{{\text{\smaller\smaller$\scriptstyle\bullet$}}}

\DeclareMathOperator{\Hom}{Hom}
\DeclareMathOperator{\Ext}{Ext}

\DeclareMathOperator{\Spec}{Spec}
\DeclareMathOperator{\pd}{pd}
\DeclareMathOperator{\coker}{coker}

\newcommand{\Modl}{{\operatorname{\mathsf{--Mod}}}}

\newcommand{\qcoh}{{\operatorname{\mathsf{--qcoh}}}} 

\newcommand{\Sets}{\mathsf{Sets}}

\newcommand{\Hot}{\mathsf{Hot}}
\newcommand{\Com}{\mathsf{Com}}
\newcommand{\Ac}{\mathsf{Ac}}
\newcommand{\Fil}{\mathsf{Fil}}
\newcommand{\Ab}{\mathsf{Ab}}

\newcommand{\fl}{\mathsf{fl}}
\newcommand{\hfl}{\mathsf{hfl}}
\renewcommand{\cot}{\mathsf{cot}}

\newcommand{\A}{\mathcal A}
\newcommand{\B}{\mathcal B}
\newcommand{\C}{\mathcal C}

\newcommand{\F}{\mathcal F}
\newcommand{\G}{\mathcal G}

\newcommand{\M}{\mathcal M}
\newcommand{\N}{\mathcal N}
\renewcommand{\O}{\mathcal O}
\renewcommand{\P}{\mathcal P}

\newcommand{\sA}{\mathsf A}
\newcommand{\sB}{\mathsf B} 
\newcommand{\sC}{\mathsf C}
\newcommand{\sD}{\mathsf D}
\newcommand{\sE}{\mathsf E}
\newcommand{\sK}{\mathsf K}
\newcommand{\sL}{\mathsf L}
\newcommand{\sP}{\mathsf P}
\newcommand{\sS}{\mathsf S}

\newcommand{\boZ}{\mathbb Z}

\newcommand{\Section}[1]{\bigskip\section{#1}\medskip}
\theoremstyle{plain}
\newtheorem{thm}{Theorem}[section]

\newtheorem{lem}[thm]{Lemma}
\newtheorem{prop}[thm]{Proposition}
\newtheorem{cor}[thm]{Corollary}
\theoremstyle{definition}

\newtheorem{rem}[thm]{Remark}

\begin{document}

\title{Flat quasi-coherent sheaves as directed colimits, \\
and quasi-coherent cotorsion periodicity}

\author{Leonid Positselski}

\address{Leonid Positselski, Institute of Mathematics, Czech Academy
of Sciences \\ \v Zitn\'a~25, 115~67 Praha~1 \\ Czech Republic}

\email{positselski@math.cas.cz}

\author{Jan \v S\v tov\'\i\v cek}

\address{Jan {\v S}{\v{t}}ov{\'{\i}}{\v{c}}ek, Charles University,
Faculty of Mathematics and Physics, Department of Algebra,
Sokolovsk\'a 83, 186 75 Praha, Czech Republic}

\email{stovicek@karlin.mff.cuni.cz}

\begin{abstract}
 We show that every flat quasi-coherent sheaf on a quasi-compact
quasi-separated scheme is a directed colimit of locally countably
presentable flat quasi-coherent sheaves.
 More generally, the same assertion holds for any countably
quasi-compact, countably quasi-separated scheme.
 Moreover, for three categories of complexes of flat quasi-coherent
sheaves, we show that all complexes in the category can be obtained as
directed colimits of complexes of locally countably presentable flat
quasi-coherent sheaves from the same category.
 In particular, on a quasi-compact semi-separated scheme, every
flat quasi-coherent sheaf is a directed colimit of flat quasi-coherent
sheaves of finite projective dimension.
 In the second part of the paper, we discuss cotorsion periodicity in
category-theoretic context, generalizing an argument of Bazzoni,
Cort\'es-Izurdiaga, and Estrada.
 As the main application, we deduce the assertion that any
cotorsion-periodic quasi-coherent sheaf on a quasi-compact
semi-separated scheme is cotorsion.
\end{abstract}

\maketitle

\tableofcontents

\section{Introduction}
\medskip

 A classical theorem of Govorov and Lazard~\cite{Gov,Laz} tells that
any flat module over an associative ring is a directed colimit of 
finitely generated free modules.
 Several attempts have been made in the literature to obtain
an analogue of this result for quasi-coherent sheaves on
schemes~\cite[Theorem~5.4]{CB}, \cite{EGO}, but the theory of flat and
projective quasi-coherent sheaves is more complicated than for modules.

 First of all, on a nonaffine scheme $X$, there are usually \emph{no}
nonzero projective objects in the abelian category of quasi-coherent
sheaves $X\qcoh$.
 Instead of projective objects, one might want to consider
the \emph{locally projective} quasi-coherent sheaves, otherwise known
as (finite- or infinite-dimensional) \emph{vector bundles} over~$X$.
 A celebrated theorem of Raynaud and Gruson~\cite[\S\,II.3.1]{RG},
\cite{Pe} tells that projectivity of modules over commutative rings
is indeed a local property; so the notion of a locally projective
quasi-coherent sheaf on $X$ is well-defined.

 A Noetherian scheme with enough (finite-dimensional) vector bundles
is said to have the \emph{resolution property}.
 Noetherian schemes (and stacks) having the resolution property were
characterized by several equivalent conditions in the paper~\cite{To}.
 In particular, \cite[Proposition~1.3]{To} tells that any Noetherian
scheme with the resolution property is semi-separated.
 It seems to be an open problem whether there exists a semi-separated
Noetherian scheme (or even a separated scheme of finite type over
a field) \emph{not} having the resolution property.

 On the other hand, it is known~\cite[Section~2.4]{M-n},
\cite[Lemma~A.1]{EP} that there are enough flat quasi-coherent sheaves
on any quasi-compact semi-separated scheme.
 Conversely, any quasi-compact quasi-separated scheme with enough flat
quasi-coherent sheaves is semi-separated~\cite[Theorem~2.2]{SS}.
 Moreover, over a quasi-compact semi-separated scheme $X$, there are
enough \emph{very flat} quasi-coherent sheaves; so any quasi-coherent
sheaf is a quotient of a very flat one~\cite[Lemma~4.1.1]{Pcosh}.
 Very flat quasi-coherent sheaves form a special class of flat
quasi-coherent sheaves locally of projective dimension at most~$1$.
 Given these results, the question about extending the Govorov--Lazard
characterization of flat modules to flat quasi-coherent sheaves on
quasi-compact semi-separated schemes naturally arises.

 Let $X$ be a quasi-compact semi-separated scheme.
 Assuming that $X$ has enough locally countably generated vector
bundles, Estrada, Guil Asensio, and Odaba\c si proved in~\cite{EGO}
that every flat quasi-coherent sheaf on $X$ can be presented as
a directed colimit of locally countably generated flat quasi-coherent
sheaves locally of projective dimension at most~$1$
\,\cite[Theorem~B or Theorem~4.9]{EGO}.

 The aim of this paper is to remove the ``enough locally countably
generated vector bundles'' assumption in the latter result, and
relax the quasi-compactness and semi-separatedness assumptions.
 We prove that, on any quasi-compact quasi-separated scheme $X$,
any flat quasi-coherent sheaf is a directed colimit of locally
countably presentable flat quasi-coherent sheaves.
 It is well-known that any countably presentable flat module over
a ring has projective dimension at most~$1$; in fact, any countably
presentable flat module is a directed colimit of a countable directed
diagram of finitely generated free modules~\cite[Corollary~2.23]{GT}.
 Thus our theorem is indeed a generalization of the one of
Estrada, Guil Asensio, and Odaba\c si.

 More generally, we say that a scheme $X$ is \emph{countably
quasi-compact} and \emph{countably quasi-separated} if $X$ can be
covered by at most countably many affine open subschemes, and
the intersection of any two affine open subschemes in $X$ can be
so covered as well.
 For such schemes, we also prove that any flat quasi-coherent sheaf
is a directed colimit of locally countably presentable flat
quasi-coherent sheaves.
 It follows from our result that any morphism from a locally countably
presentable quasi-coherent sheaf on $X$ to a flat quasi-coherent sheaf
factorizes through a locally countably presentable flat quasi-coherent
sheaf.
 Thus ``a~locally countably presentable version of the Govorov--Lazard
theory of flatness'' holds over countably quasi-compact,
countably semi-separated schemes.

 We also briefly turn our attention to arbitrary (not necessarily
flat) quasi-coherent sheaves.
 Our methods allow one to prove that the abelian category $X\qcoh$ is
locally $\kappa$\+presentable for any regular cardinal~$\kappa$ and any
$\kappa$\+quasi-compact, $\kappa$\+quasi-separated scheme~$X$.
 The $\kappa$\+presentable objects of $X\qcoh$ are precisely all
the locally $\kappa$\+presentable quasi-coherent sheaves on $X$ in
this case.
 This provides a stronger version of a well-known result of
Gabber~\cite[Lemma Tag~077N]{SP}, which established local
$\kappa$\+generatedness of $X\qcoh$, as well as an extension of
a theorem of Grothendieck~\cite[Corollaire~I.6.9.12]{EGA1} from
the case of $\kappa=\aleph_0$ to that of an arbitrary regular
cardinal~$\kappa$.

 Moreover, we discuss three categories of complexes of flat
quasi-coherent sheaves on a scheme~$X$: arbitrary complexes of flats,
acyclic complexes of flats with flat sheaves of cocycles, and
homotopy flat complexes of flats (otherwise known as ``semiflat
complexes'').
 In all the three cases, we show that all complexes from the respective
category are directed colimits of complexes of locally countably
presentable sheaves from the same category.
 The only caveat is that, in the case of homotopy flat complexes, we
need the scheme to be semi-separated.

 Concerning the homotopy flat complexes of flats, let us mention here
that a theorem of Christensen and Holm~\cite[Theorem~1.1]{CH} describes
the homotopy flat complexes of flat modules over an arbitrary ring $R$
as the directed colimits of bounded complexes of finitely generated
projective $R$\+modules.
 Our result provides a countable version of the Christensen--Holm
theorem for quasi-coherent sheaves. 

 Let us say a few words about the proofs.
 The proofs of the main results of this paper are based on category
theory.
 Let $\kappa$~be a regular cardinal and $\lambda<\kappa$ be a smaller
infinite cardinal.
 A very general category-theoretic principle going back to
Ulmer~\cite[Section~3]{Ulm} claims that $\kappa$\+accessible categories
with directed colimits of $\lambda$\+indexed chains are preserved by
various ``limit-type'' category-theoretic constructions.
 These include the pseudopullback~\cite{CR,RR}, as well as the inserter,
equifier, and isomorpher~\cite{Pacc} with respect to functors preserving
$\kappa$\+directed colimits and colimits of $\lambda$\+indexed chains
and taking $\kappa$\+presentable objects to
$\kappa$\+presentable objects.
 Moreover, the simplest expected description of the full subcategory
of $\kappa$\+presentable objects in the resulting $\kappa$\+accessible
category holds true.

 The category of flat modules over any ring $R$, as well as various
categories of complexes of flat $R$\+modules we are interested in,
is $\aleph_1$\+accessible with directed colimits of countable chains.
 That is what makes the proofs in this paper possible.

 In the case of a quasi-compact semi-separated scheme, we explain that
having finite projective dimension is a local property of
quasi-coherent sheaves.
 In other words, every quasi-coherent sheaf locally of finite
projective dimension on a quasi-compact semi-separated scheme $X$
has finite projective dimension as an object of the abelian category
of quasi-coherent sheaves $X\qcoh$.
 Therefore, it follows from our result that every flat quasi-coherent
sheaf on $X$ is a directed colimit of flat quasi-coherent sheaves
of finite projective dimension in $X\qcoh$.

 In the second part of the paper, we discuss \emph{cotorsion
periodicity}, first in an abstract category-theoretic setting and
then specifically for quasi-coherent sheaves on quasi-compact
semi-separated schemes, based on the results of the first part of
the paper.
 The subject of \emph{periodicity theorems} in homological algebra
goes back to the now-classical paper of Benson and Goodearl~\cite{BG},
where it was shown that any flat projective-periodic module is
projective.
 This means that if $0\rarrow F\rarrow P\rarrow F\rarrow0$ is
a short exact sequence of flat modules over an associative ring $R$,
and the module $P$ is projective, then the module $F$ is
projective as well~\cite[Theorem~2.5]{BG}.
 The original argument in~\cite{BG} was module-theoretic.

 Later the same result was independently obtained (and strengthened)
by Neeman~\cite{Neem}, who used a very different technique of
complexes of modules and their morphisms up to cochain homotopy.
 In Neeman's formulation, the flat/projective periodicity theorem tells
that, in any (unbounded) acyclic complex of projective modules over
an associative ring with flat modules of cocycles, the modules of
cocycles are actually
projective~\cite[Remark~2.15 and Theorem~8.6]{Neem}.

 Cotorsion periodicity was discovered by Bazzoni, Cort\'es-Izurdiaga,
and Estrada~\cite{BCE}, who showed that any cotorsion-periodic
module is cotorsion~\cite[Theorem~1.2(2), Proposition~4.8(2), or
Theorem~5.1(2)]{BCE}.
 This theorem tells that if $0\rarrow M\rarrow C\rarrow M\rarrow0$ is
a short exact sequence of modules and $C$ is a cotorsion module, then
the module $M$ is cotorsion as well.
 In other words, this means that, in any acyclic complex of cotorsion
modules over an associative ring, the modules of cocycles are
also cotorsion.
 Here a module $C$ over a ring $R$ is said to be \emph{cotorsion} if
$\Ext^1_R(F,C)=0$ for all flat $R$\+modules~$F$.
 The reader can find a general discussion of periodicity theorems in
the introduction to the paper~\cite{BHP} and in
the preprint~\cite[Sections~7.8 and~7.10]{Pphil}.

 We work out a category-theoretic version of the argument of
Bazzoni, Cort\'es-Izurdiaga, and Estrada for cotorsion periodicity,
replacing the considerations of purity with ones of flatness,
which relaxes and simplifies the assumptions.
 Then, as an application of our result about presenting flat
quasi-coherent sheaves as directed colimits, we show that
the cotorsion periodicity holds for quasi-coherent sheaves on
any quasi-compact semi-separated scheme~$X$.
 This means that, in any (unbounded) acyclic complex of cotorsion
quasi-coherent sheaves on $X$, the sheaves of cocycles are cotorsion.
 We recall that a quasi-coherent sheaf $\C$ on $X$ is said to be
\emph{cotorsion} if $\Ext^1_{X\qcoh}(\F,\C)=0$ for all flat
quasi-coherent sheaves~$\F$ \,\cite{EE},
and that this property of being cotorsion cannot be in general
detected locally on modules of sections (cf.\ the discussion
in~\cite[Section~8]{Pphil}).

 Let us say a few more words about the connection between periodicity
and the cocycles in acyclic complexes.
 This is a simple observation going back to~\cite[proof of
Proposition~7.6]{CH} and~\cite[Propositions~1 and~2]{EFI}.
 Let $\sK$ be an abelian category with countable products and $\sB
\subset\sK$ be a class of objects closed under direct summands and
countable products.
 Assuming that the countable product functor in $\sK$ is exact,
it is known that the following two properties are equivalent:
every $\sB$\+periodic object of $\sK$ belongs to $\sB$ if and only
if, in every acyclic complex in $\sK$ with the terms in $\sB$,
the objects of cocycles belong to~$\sB$.
 In this paper, we generalize this result by working out
an additional assumption on an abelian (or exact) category $\sK$
and its full subcategory $\sB$ under which this equivalence of
two properties holds even if the countable products in $\sK$ are
\emph{not} exact.

 As a final application, we deduce the assertion that the derived
category of the abelian category of quasi-coherent sheaves on
a quasi-compact semi-separated scheme $X$ is equivalent to the derived
category of the exact category of cotorsion quasi-coherent sheaves,
$\sD(X\qcoh^\cot)\simeq\sD(X\qcoh)$.
 Thus the derived category of quasi-coherent sheaves can be described
in terms of cotorsion sheaves.

 Sections~\ref{preliminaries-on-categories-secn}\+-%
\ref{qcoh-finite-projdim-secn} form the first part of the paper.
 Section~\ref{preliminaries-on-categories-secn} supplies preliminary
material on accessibility of various categories of modules and
complexes of modules over rings.
 The countable Govorov--Lazard theorem for flat quasi-coherent sheaves
on quasi-compact quasi-separated schemes is proved in
Section~\ref{flat-sheaves-on-qcomp-qsep-schemes-secn}.
 This result is extended to countably quasi-compact, countably
quasi-separated schemes in
Section~\ref{flat-sheaves-on-countably-qc-qs-secn}.
 Our discussion of the local presentability rank of the abelian
category of quasi-coherent sheaves can be also found in
Section~\ref{flat-sheaves-on-countably-qc-qs-secn}.
 Three classes of complexes of flat quasi-coherent sheaves are
discussed in Section~\ref{complexes-of-flats-secn}.
 The connection between the local and global projective dimensions
of quasi-coherent sheaves is explained in
Section~\ref{qcoh-finite-projdim-secn}.

 The second part of the paper consists of
Sections~\ref{prelim-ext-orthogonal-secn}\+-%
\ref{qcoh-cotorsion-periodicity-secn}.
 Section~\ref{prelim-ext-orthogonal-secn} contains the preliminaries
for the second part.
 An important technical proposition about closure properties of classes
of objects in exact categories with exact directed colimits (such as
the categories of flat modules or flat sheaves) is proved
in Section~\ref{exact-categs-exact-dirlims-secn}.
 Based on that, we flesh out our category-theoretic approach to
cotorsion periodicity in
Section~\ref{categorical-cotorsion-periodicity-secn}.
 The cotorsion periodicity for quasi-coherent sheaves is established
in the final Section~\ref{qcoh-cotorsion-periodicity-secn} by combining
the results of the first and second parts of the paper.

\subsection*{Acknowledgement}
 The authors are grateful to Michal Hrbek, Silvana Bazzoni,
Souvik Dey, and Alexander Sl\'avik for helpful discussions.
 We also want to thank the anonymous referee for his insightful reading
of our manuscript and several helpful suggestions.
 In particular, the problem described in
Remark~\ref{prolongation-problem-remark} was spotted by the referee.
 This work is supported by the GA\v CR project 20-13778S.
 The first-named author is also supported by research plan
RVO:~67985840.

\Section{Preliminaries on Categories of Flat Modules and Complexes}
\label{preliminaries-on-categories-secn}

 This paper is written in category-theoretic language, and accessible
categories play an important role.
 We use the book~\cite{AR} as the main background reference source on
accessible categories.

 Let $\kappa$~be a regular cardinal.
 We refer to~\cite[Definition~1.4, Theorem and Corollary~1.5,
Definition~1.13(1), and Remark~1.21]{AR} for a relevant discussion of
\emph{$\kappa$\+directed posets} vs.\ \emph{$\kappa$\+filtered small
categories}, and colimits indexed by these.
 For the purposes of this paper, we are mainly interested in two cases:
$\kappa=\aleph_0$ or $\kappa=\aleph_1$.

 We denote by $\Sets$ the category of sets and by $\Ab$ the category
of abelian groups.

 Let $\sA$ be a category with $\kappa$\+directed colimits.
 An object $P\in\sA$ is said to be \emph{$\kappa$\+presentable}
if the functor $\Hom_\sA(P,{-})\:\sA\rarrow\Sets$ preserves
$\kappa$\+directed colimits~\cite[Definition~1.13(2)]{AR}.
 When the category $\sA$ is additive, this condition is equivalent to
the functor $\Hom_\sA(P,{-})\:\sA\rarrow\Ab$ preserving
$\kappa$\+directed colimits.
 Given a category $\sA$ with $\kappa$\+directed colimits,
we denote by $\sA_{<\kappa}\subset\sA$ the full subcategory of
$\kappa$\+presentable objects in~$\sA$.

 A category $\sA$ with $\kappa$\+directed colimits is said to be
\emph{$\kappa$\+accessible} if there is a set of $\kappa$\+presentable
objects $\sP\subset \sA$ such that all the objects of $\sA$ are
$\kappa$\+directed colimits of objects from~$\sP$
\,\cite[Definition~2.1]{AR}.
 If this is the case, then the $\kappa$\+presentable objects of $\sA$
are precisely all the retracts of the objects from~$\sP$.
 A $\kappa$\+accessible category where all colimits exist is called
\emph{locally $\kappa$\+presentable}~\cite[Definition~1.17]{AR}.

 $\aleph_0$\+presentable objects are called \emph{finitely
presentable}~\cite[Definition~1.1]{AR}, locally $\aleph_0$\+presentable
categories are called \emph{locally finitely
presentable}~\cite[Definition~1.9]{AR}, and
$\aleph_0$\+accessible categories are called \emph{finitely
accessible}~\cite[Remark~2.2(1)]{AR}.

 We refer to~\cite[Theorem~2.11 and Definition~2.12]{AR} for
the definition of the partial order~$\triangleleft$ on the class
of all regular cardinals.
 For our purposes, it is only important that $\aleph_0\triangleleft
\kappa$ for every uncountable regular cardinal~$\kappa$
\,\cite[Example~2.13(1)]{AR}.

 A set is said to be \emph{$\kappa$\+small} if it has less than~$\kappa$
elements.
 A category is said to be \emph{$\kappa$\+small} if it has less
than~$\kappa$ morphisms.
 The \emph{$\kappa$\+small colimits} are the ones indexed by
$\kappa$\+small posets or categories.

\begin{prop} \label{sharply-smaller-cardinal}
 Let $\lambda\triangleleft\kappa$ be a pair of regular cardinals, and
let\/ $\sA$ be a\/ $\lambda$\+accessible category.
 Then $\sA$ is also a $\kappa$\+accessible category.
 The $\kappa$\+presentable objects of\/ $\sA$ are precisely all
the retracts of the $\kappa$\+small\/ $\lambda$\+directed colimits
of the\/ $\lambda$\+presentable objects of\/~$\sA$.
\end{prop}

\begin{proof}
 The first assertion is~\cite[Theorem~2.11(i)]{AR} (for $\mu=\kappa$).
 The second assertion follows from the proof
of~\cite[Theorem~2.11\,(iv)\,$\Rightarrow$\,(i)]{AR}.
\end{proof}

 Let $R$ be an associative ring.
 We are interested in the additive category of flat $R$\+modules
$R\Modl_\fl$, but we start with a discussion of the ambient abelian
category of arbitrary $R$\+modules $R\Modl$.
 The following lemma is well-known and standard.

\begin{lem} \label{modules-locally-presentable}
 For any ring $R$, the abelian category $R\Modl$ is locally finitely
presentable, and consequently, locally $\kappa$\+presentable for
every regular cardinal~$\kappa$.
 The $\kappa$\+presentable objects of $R\Modl$ are precisely all
the $R$\+modules that can be represented as the cokernel of a morphism
of free $R$\+modules with less than~$\kappa$ generators.  \qed
\end{lem}

 We will use the terminology \emph{countably presentable} for
$\aleph_1$\+presentable $R$\+modules.
 So an $R$\+module is said to be countably presentable if it can be
represented as the cokernel of a morphism of free $R$\+modules with
at most countable sets of generators.

\begin{prop} \label{flat-modules-accessible}
 For any ring $R$, the additive ctategory $R\Modl_\fl$ is finitely
accessible, and consequently, $\kappa$\+accessible for every
regular cardinal~$\kappa$.
 The $\kappa$\+presentable objects of $R\Modl_\fl$ are precisely all
the flat $R$\+modules that are $\kappa$\+presentable in $R\Modl$.
\end{prop}

\begin{proof}
 Obviously, the full subcategory $R\Modl_\fl$ is closed under directed
colimits in $R\Modl$.
 Hence any object of $R\Modl_\fl$ that is $\kappa$\+presentable in
$R\Modl$ is also $\kappa$\+presentable in $R\Modl_\fl$.

 The Govorov--Lazard characterization of flat
$R$\+modules~\cite{Gov,Laz}, \cite[Corollary~2.22]{GT} tells that
the flat $R$\+modules are precisely the directed colimits of finitely
generated projective $R$\+modules.
 As finitely generated projective $R$\+modules are finitely presentable,
it follows that the category $R\Modl_\fl$ is finitely accessible.

 It remains to refer to Proposition~\ref{sharply-smaller-cardinal},
which describes the $\kappa$\+presentable objects of $R\Modl_\fl$
as the direct summands of $\kappa$\+small directed colimits of finitely
generated projective $R$\+modules.
 Obviously, all such modules are $\kappa$\+presentable in $R\Modl$.
\end{proof}

 Now we pass to complexes of flat $R$\+modules.
 Given an additive category $\sA$, we denote by $\Com(\sA)$
the category of cochain complexes in~$\sA$.

\begin{prop} \label{complexes-of-flats-acessible}
 The category of complexes of flat $R$\+modules\/ $\Com(R\Modl_\fl)$
is\/ $\aleph_1$\+accessible.
 The\/ $\aleph_1$\+presentable objects of\/ $\Com(R\Modl_\fl)$ are
precisely all the complexes of countably presentable flat $R$\+modules.
\end{prop}

\begin{proof}
 This is a corollary of Proposition~\ref{flat-modules-accessible}
and~\cite[Theorem~6.2]{Pacc}.
 See~\cite[Corollary~10.4]{Pacc}.
\end{proof}

\begin{lem} \label{short-exact-seq-of-flats-accessible}
 The category\/ $\Ac^3(R\Modl_\fl)$ of short exact sequences of flat
$R$\+modules is\/ $\aleph_1$\+accessible.
 The\/ $\aleph_1$\+presentable objects of\/ $\Ac^3(R\Modl_\fl)$ are
precisely all the short exact sequences of countably presentable
flat $R$\+modules.
\end{lem}

\begin{proof}
 It is instructive to notice that the kernel of any surjective morphism
of countably presentable flat $R$\+modules is countably
presentable~\cite[Lemma~4.1]{Pflcc}, \cite[Corollary~10.12]{Pacc},
\cite[Corollary~4.7]{Plce}.
 Hence, if two terms of a short exact sequence of flat $R$\+modules are
countably presentable, then so is the third term.

 The assertion of the lemma can be found in~\cite[Corollary~10.13]{Pacc}
or~\cite[first assertion of Proposition~4.6]{Plce}.
\end{proof}

\begin{prop} \label{acyclic-flat-cocycles-accessible}
 The category\/ $\Ac_\fl(R\Modl_\fl)$ of acyclic complexes of flat
$R$\+modules with flat $R$\+modules of cocycles is\/
$\aleph_1$\+accessible.
 The\/ $\aleph_1$\+presentable objects of\/ $\Ac_\fl(R\Modl_\fl)$ are
precisely all the acyclic complexes of countably presentable flat
$R$\+modules with flat modules of cocycles.
\end{prop}

\begin{proof}
 This is a corollary of Lemma~\ref{short-exact-seq-of-flats-accessible}
and~\cite[Remark~5.2]{Pacc}.
 See~\cite[Corollary~10.14]{Pacc}.
\end{proof}

 A complex of left $R$\+modules $F^\bu$ is said to be
\emph{homotopy flat}~\cite{Spal} if, for every acyclic complex of
right $R$\+modules $A^\bu$, the complex of abelian groups
$A^\bu\ot_RF^\bu$ is acyclic.
 In this paper, we are interested in homotopy flat complexes of
\emph{flat} $R$\+modules; in the terminology of the paper~\cite{CH},
these are called \emph{semi-flat} complexes.
 We denote the full subcategory of homotopy flat complexes of flat
$R$\+modules by $\Com(R\Modl_\fl)_\hfl\subset\Com(R\Modl)$.

\begin{prop} \label{homotopy-flats-accessible}
 The category\/ $\Com(R\Modl_\fl)_\hfl$ of homotopy flat complexes of
flat $R$\+modules is finitely accessible, and consequently,
$\kappa$\+accessible for all regular cardinals~$\kappa$.
 The finitely presentable objects of $\Com(R\Modl_\fl)_\hfl$ are
precisely all the bounded complexes of finitely generated projective
$R$\+modules.
 For any uncountable regular cardinal~$\kappa$,
the $\kappa$\+presentable objects of\/ $\Com(R\Modl_\fl)_\hfl$ are
precisely all the homotopy flat complexes of $\kappa$\+presentable
flat $R$\+modules.
\end{prop}

\begin{proof}
 It is easy to see that all bounded complexes of finitely presentable
$R$\+modules are finitely presentable in $\Com(R\Modl)$.
 Since the full subcategory $\Com(R\Modl_\fl)_\hfl$ is closed under
directed colimits in $\Com(R\Modl)$, all bounded complexes of finitely
generated projective $R$\+modules are finitely presentable in
$\Com(R\Modl_\fl)_\hfl$.

 Similarly, for any uncountable regular cardinal~$\kappa$,
all complexes of $\kappa$\+presentable $R$\+modules are
$\kappa$\+presentable in $\Com(R\Modl)$.
 Consequently, all homotopy flat complexes of $\kappa$\+presentable
flat $R$\+modules are $\kappa$\+presentable in $\Com(R\Modl_\fl)_\hfl$.

 The nontrivial assertion that all homotopy flat complexes of flat
$R$\+modules are directed colimits of bounded complexes of finitely
generated projective $R$\+modules is the result
of~\cite[Theorem~1.1]{CH}.
 The description of $\kappa$\+presentable objects
of $\Com(R\Modl_\fl)_\hfl$ follows by virtue of
Proposition~\ref{sharply-smaller-cardinal}.
\end{proof}

\Section{Flat Quasi-Coherent Sheaves as Directed Colimits~I}
\label{flat-sheaves-on-qcomp-qsep-schemes-secn}

 We start with a lemma showing that a ``locally countably
presentable quasi-coherent sheaf'' is a well-defined notion.

\begin{lem} \label{locality-generatedness-presentability}
 Let $U=\bigcup_\alpha U_\alpha$ be an affine scheme covered by
a finite number of affine open subschemes.
 Let $M$ be an $\O(U)$\+module.
 Then \par
\textup{(a)} the $\O(U)$\+module $M$ is finitely generated if and
only if, for every index~$\alpha$, the $\O(U_\alpha)$\+module
$\O(U_\alpha)\ot_{\O(U)}M$ is finitely generated; \par
\textup{(b)} the $\O(U)$\+module $M$ is finitely presentable if and
only if, for every index~$\alpha$, the $\O(U_\alpha)$\+module
$\O(U_\alpha)\ot_{\O(U)}M$ is finitely presentable; \par
\textup{(c)} the $\O(U)$\+module $M$ is countably generated if and
only if, for every index~$\alpha$, the $\O(U_\alpha)$\+module
$\O(U_\alpha)\ot_{\O(U)}M$ is countably generated; \par
\textup{(d)} the $\O(U)$\+module $M$ is countably presentable if and
only if, for every index~$\alpha$, the $\O(U_\alpha)$\+module
$\O(U_\alpha)\ot_{\O(U)}M$ is countably presentable.
\end{lem}

\begin{proof}
 All the claims follow from the observations that
\begin{itemize}
\item the localization/tensor product functors
$\O(U_\alpha)\ot_{\O(U)}{-}$ are exact;
\item for any $\O(U)$\+module $M$, any index~$\alpha$, and any finitely
generated $\O(U_\alpha)$\+sub\-module $L_\alpha\subset
\O(U_\alpha)\ot_{\O(U)}M$, there exists a finitely generated
$\O(U)$\+submod\-ule $L\subset M$ such that $L_\alpha\subset
\O(U_\alpha)\ot_{\O(U)}L\subset\O(U_\alpha)\ot_{\O(U)}M$;
\item if, for a given $\O(U)$\+module $M$, one has
$\O(U_\alpha)\ot_{\O(U)}M=0$ for all the indices~$\alpha$, then $M=0$;
\item the kernel of a surjective morphism from a finitely generated
module to a finitely presentable one is finitely generated; and
similarly, the kernel of a surjective morphism from a countably
generated module to a countably presentable one is countably generated.
\end{itemize}
 We leave details to the reader.
\end{proof}

 A quasi-coherent sheaf $\M$ on a scheme $X$ is said to be
\emph{locally finitely generated} (respectively, \emph{locally finitely
presentable}, \emph{locally countably generated}, or \emph{locally
countably presentable}) if, for every affine open subscheme
$U\subset X$, the $\O(U)$\+module $\M(U)$ is finitely generated (resp.,
finitely presentable, countably generated, or countably presentable).
 Lemma~\ref{locality-generatedness-presentability} tells that it 
suffices to check these properties for affine open subschemes $U_\alpha$
belonging to any chosen affine open covering
$X=\bigcup_\alpha U_\alpha$ of a scheme~$X$.

 The next lemma allows to glue quasi-coherent sheaves on a scheme $X$
from a covering by two open subschemes.

\begin{lem} \label{glue-from-a-cover-by-two}
 Let $X=U\cup V$ be a covering of a scheme $X$ by two open subschemes
$U$, $V\subset X$.
 Then the category of quasi-coherent sheaves $\M$ on $X$ is equivalent
to the category formed by the following sets of data:
\begin{enumerate}
\item a quasi-coherent sheaf $\M_U$ on $U$ and a quasi-coherent sheaf
$\M_V$ on $V$ are given;
\item an isomorphism between the two restrictions of quasi-coherent
sheaves to an open subscheme $(\M_U)|_{U\cap V}\simeq(\M_V)|_{U\cap V}$
is given.
\end{enumerate}
 Here morphisms between the sets of data are defined in the obvious way.
\end{lem}

\begin{proof}
 Standard and straightforward.
\end{proof}

 Let $X$ be a scheme.
 A quasi-coherent sheaf $\F\in X\qcoh$ is said to be \emph{flat}
if, for every affine open subscheme $U\subset X$, the $\O(U)$\+module
$\F(U)$ is flat.
 It suffices to check this condition for the affine open subschemes
$U\subset X$ belonging to any fixed affine open covering of~$X$.
 If the scheme $X$ is quasi-separated, then a quasi-coherent sheaf
$\F$ on $X$ is flat if and only if the tensor product functor
${-}\ot_{\O_X}\F\:X\qcoh\rarrow X\qcoh$ is exact (as a functor on
the abelian category $X\qcoh$).

 We denote the abelian category of quasi-coherent sheaves on
a scheme $X$ by $X\qcoh$ and the full subcategory of flat
quasi-coherent sheaves by $X\qcoh_\fl\subset X\qcoh$.

\begin{rem} \label{prolongation-problem-remark}
 It is obvious that, for any scheme $X$ and any flat quasi-coherent
sheaf $\F\in X\qcoh$, the tensor product functor
${-}\ot_{\O_X}\F\:X\qcoh\rarrow X\qcoh$ is exact.
 Conversely, let $\F$ be a quasi-coherent sheaf on $X$ for which
the functor ${-}\ot_{\O_X}\F$ is exact on $X\qcoh$.
 Then, in order to deduce flatness of the $\O(U)$\+module $\F(U)$
for an affine open subscheme $U\subset X$, one would need to be able
to extend any quasi-coherent sheaf $\M_U$ on $U$ with a quasi-coherent
subsheaf $\N_U\subset\M_U$ to a quasi-coherent sheaf $\M$ on $X$ with
a quasi-coherent subsheaf $\N\subset\M$.

 This problem of prolongation of quasi-coherent sheaves is easily
solvable for a quasi-compact open immersion morphism $j\:U\rarrow X$;
one can simply put $\M=j_*\M_U$ and $\N=j_*\N_U$
(see~\cite[Proposition~I.6.9.2]{EGA1} or~\cite[Lemma Tag~01PE]{SP}).
 We do \emph{not} know how to solve this problem otherwise.
 It would be even sufficient (for the purposes of characterizing
flat quasi-coherent sheaves $\F$ as above) to assume that $\M_U=\O_U$;
then one can take $\M=\O_X$.
 Still we do \emph{not} know how to extend a quasi-coherent subsheaf
$\N_U\subset\O_U$ to a quasi-coherent subsheaf $\N\subset\O_X$ (i.~e.,
how to extend a quasi-coherent sheaf of ideals from an open subscheme)
without assuming the open immersion morphism $j\:U\rarrow X$ to be
quasi-compact.
 Hence the assumption of quasi-separatedness of the scheme $X$ in
the paragraph preceding this remark.
 (Cf.\ Lemma~\ref{homotopy-flatness-locality-lemma}
and Corollary~\ref{homotopy-flatness-locality-cor} below.)
 We are grateful to the anonymous referee for bringing this issue
to our attention.
\end{rem}

 Let $\sA$, $\sB$, and $\sC$ be three categories, and let
$F\:\sA\rarrow\sC$ and $G\:\sB\rarrow\sC$ be two functors.
 The \emph{pseudopullback} of the two functors $F$ and $G$ is
defined as the category $\sD$ of all triples $(A,B,\theta)$, where
$A\in\sA$ and $B\in\sB$ are objects, and $\theta\:F(A)\simeq G(B)$
is an isomorphism in~$\sC$.
 It is important for us that Lemma~\ref{glue-from-a-cover-by-two}
represents the category $X\qcoh$ of quasi-coherent sheaves on $X$ as
a pseudopullback of the categories $U\qcoh$ and $V\qcoh$ over
$(U\cap\nobreak V)\qcoh$.

 Given a category $\sA$ and a ordinal~$\alpha$, by
a \emph{$\alpha$\+indexed chain} (of objects and morphisms) in $\sA$ 
we mean a functor $\alpha\rarrow\sA$, where the directed set~$\alpha$
is viewed as a category in the usual way.
 So an $\alpha$\+indexed chain is a directed diagram
$(A_i\to A_j)_{0\le i<j<\alpha}$ in~$\sA$.
 In the following proposition, a cardinal~$\lambda$ is considered
as an ordinal.

\begin{prop} \label{pseudopullback-prop}
 Let $\kappa$~be a regular cardinal, and $\lambda<\kappa$ be a smaller
infinite cardinal.
 Let\/ $\sA$, $\sB$, and\/ $\sC$ be $\kappa$\+accessible categories in
which the colimits of all $\lambda$\+indexed chains exist.
 Let $F\:\sA\rarrow\sC$ and $G\:\sB\rarrow\sC$ be two functors
preserving $\kappa$\+directed colimits and colimits of
$\lambda$\+indexed chains, and taking $\kappa$\+presentable objects to
$\kappa$\+presentable objects.
 Then the pseudopullback\/ $\sD$ of the pair of functors $F$ and $G$ is
a $\kappa$\+accessible category, and the $\kappa$\+presentable objects
of\/ $\sD$ are precisely all the objects $(A,B,\theta)\in\sD$ with
$A\in\sA_{<\kappa}$ and $B\in\sB_{<\kappa}$.
\end{prop}

\begin{proof}
 This is Pseudopullback Theorem~\cite[Theorem~2.2]{RR}, based
on~\cite[proof of Proposition~3.1]{CR}.
 The result goes back to the unpublished preprint~\cite[Remark~3.2(I),
Theorem~3.8, Corollary~3.9, and Remark~3.11(II)]{Ulm}.
 See also~\cite[Corollary~5.1]{Pacc}.
\end{proof}

 Now we can prove a more restricted version of the main result of
this paper.

\begin{thm} \label{flat-qcoh-on-qcomp-qsep}
 Let $X$ be a quasi-compact quasi-separated scheme.
 Then the category of flat quasi-coherent sheaves $X\qcoh_\fl$ is\/
$\aleph_1$\+accessible.
 The $\aleph_1$\+presentable objects of $X\qcoh_\fl$ are precisely all
the locally countably presentable flat quasi-coherent sheaves on~$X$.
 Hence every flat quasi-coherent sheaf on $X$ is
an\/ $\aleph_1$\+directed colimit of locally countably presentable
flat quasi-coherent sheaves.
\end{thm}

\begin{proof}
 The proof proceeds in two steps.
 On the first step, we assume that the scheme $X$ is semi-separated.
 Then we argue by induction on the number~$d$ of affine open
subschemes in an affine open covering $X=\bigcup_{\alpha=1}^d U_\alpha$
of the scheme~$X$.
 The affine case $d=1$ holds by
Proposition~\ref{flat-modules-accessible}.

 For $d>1$, put $U=U_d$ and $V=\bigcup_{\alpha=1}^{d-1}U_\alpha$.
 Then $U\cap V=\bigcup_{\alpha=1}^{d-1}U_d\cap U_\alpha$ is
an open covering of the intersection $U\cap V$ by $d-1$ affine
open subschemes.
 By the induction assumption, we know the assertion of the theorem
to hold for all the three schemes $U$, $V$, and $U\cap V$.
 By Lemma~\ref{glue-from-a-cover-by-two} (restricted to the full
subcategories of flat quasi-coherent sheaves), the category
$\sD=X\qcoh_\fl$ is the pseudopullback of the pair of categories
$\sA=U\qcoh_\fl$ and $\sB=V\qcoh_\fl$ over the category
$\sC=(U\cap\nobreak V)\qcoh_\fl$ (with respect to the functors of
restriction of quasi-coherent sheaves to open subschemes).
 It remains to apply Proposition~\ref{pseudopullback-prop}
(for $\kappa=\aleph_1$ and $\lambda=\aleph_0$).

 On the second step, we consider an arbitrary quasi-compact
quasi-separated scheme~$X$.
 Once again, we argue by induction on the number~$d$ of affine open
subschemes in an affine open covering $X=\bigcup_{\alpha=1}^d U_\alpha$,
and put $U=U_d$ and $V=\bigcup_{\alpha=1}^{d-1}U_\alpha$.
 Then $U\cap V=\bigcup_{\alpha=1}^{d-1}U_d\cap U_\alpha$ is
a quasi-compact, semi-separated scheme (as it is an open subscheme
in an affine scheme~$U$).
 By the induction assumption and by the first step of this proof,
we know the assertion of the theorem to hold for all the three schemes
$U$, $V$, and $U\cap V$.
 Once again, it remains to apply Proposition~\ref{pseudopullback-prop}
(for $\kappa=\aleph_1$ and $\lambda=\aleph_0$) together with
Lemma~\ref{glue-from-a-cover-by-two} (restricted to the categories
of flat quasi-coherent sheaves).
\end{proof}

\Section{Flat Quasi-Coherent Sheaves as Directed Colimits~II}
\label{flat-sheaves-on-countably-qc-qs-secn}

 The following lemma spells out the gluing of quasi-coherent sheaves
from an affine open covering of an arbitrary scheme.
 Let $X$ be a scheme and $X=\bigcup_{\alpha\in\Delta} U_\alpha$ be its
affine open covering, indexed by a set of indices~$\Delta$.
 The rule $\beta\le\alpha\in\Delta$ if $U_\beta\subset U_\alpha$
defines a partial (pre)order on~$\Delta$.
 Introduce the notation $R_\alpha=\O(U_\alpha)$ for the ring of
functions on the affine scheme $U_\alpha$, \,$\alpha\in\Delta$.

\begin{lem} \label{glue-from-affine-covering}
 Let $X$ be a scheme with a chosen affine open covering as per
the notation above.
 Assume that for all $\alpha$, $\beta\in\Delta$ we have
$U_\alpha\cap U_\beta=\bigcup_{\gamma\le\alpha}^{\gamma\le\beta}
U_\gamma$.
 Then the category of quasi-coherent sheaves $\M$ on $X$ is equivalent
to the category formed by the following sets of data:
\begin{enumerate}
\item for every\/ $\alpha\in\Delta$, an $R_\alpha$\+module
$M_\alpha=\M(U_\alpha)$ is given; and
\item for every\/ $\beta\le\alpha\in\Delta$, an isomorphism of
$R_\beta$\+modules (``gluing datum'')
$$
 m_\beta^\alpha\:R_\beta\ot_{R_\alpha}M_\alpha
 \overset\simeq\lrarrow M_\beta
$$
is given
\end{enumerate}
satisfying the following equations:
\begin{enumerate}
\setcounter{enumi}{2}
\item for all\/ $\gamma\le\alpha\le\beta\in\Delta$, the triangular
diagram of isomorphisms of $R_\alpha$\+modules
$$
 \xymatrix{
  R_\gamma\ot_{R_\alpha}M_\alpha
  \ar@{=}[rr]^-{m_\gamma^\alpha}
  \ar@{=}[rd]_-{R_\gamma\ot_{R_\beta}m_\beta^\alpha\qquad}
  & & M_\gamma \\
  & R_\gamma\ot_{R_\beta}M_\beta
  \ar@{=}[ru]_-{m_\gamma^\beta}
 }
$$
is commutative.
\end{enumerate}
\end{lem}

\begin{proof}
 The rule $M_\alpha=\M(U_\alpha)$ defines a functor from $X\qcoh$
to the category of sets of data~(1\<3).
 It is clear that this functor is fully faithful.
 To recover a quasi-coherent sheaf on $X$ from a set of data~(1\<3),
notice that it suffices to define a quasi-coherent sheaf on $X$ on
affine open subschemes $V\subset X$ subordinate to any given covering
$X=\bigcup_{\alpha\in\Delta} U_\alpha$, that is, $V$ for which there
exists $\alpha\in\Delta$ with $V\subset U_\alpha$
\,\cite[Section~0.3.2]{EGA1} (cf.~\cite[Section~2]{EE}).
 If this is the case, put $\M(V)=\O(V)\ot_{R_\alpha}M_\alpha$.

 Given an affine open subscheme $W\subset V$ and an index
$\beta\in\Delta$ with $W\subset U_\beta$, an isomorphism of
$\O(W)$\+modules
$$
 \O(W)\ot_{\O(V)}\M(V)=\O(W)\ot_{R_\alpha}M_\alpha
 \simeq\O(W)\ot_{R_\beta}M_\beta=\M(W)
$$
can be constructed using the inclusion $W\subset U_\alpha\cap U_\beta$
and the full-and-faithfulness assertion above applied to the affine
open covering $U_\alpha\cap U_\beta=
\bigcup_{\gamma\le\alpha}^{\gamma\le\beta}U_\gamma$ of
the scheme $U_\alpha\cap U_\beta$.
 The point is that, denoting by $\M_\alpha$ the quasi-coherent sheaf
on $U_\alpha$ corresponding to the $R_\alpha$\+module $M_\alpha$ and
by $\M_\beta$ the quasi-coherent sheaf on $U_\beta$ corresponding to
the $R_\beta$\+module $M_\beta$, the quasi-coherent sheaves
$\M_\alpha|_{U_\alpha\cap U_\beta}$ and $\M_\beta|_{U\alpha\cap
U_\beta}$ on the scheme $U_\alpha\cap U_\beta$ are naturally
isomorphic, because so are the related sets of data indexed by
$\gamma\in\Delta$ with $\gamma\le\alpha$, \,$\gamma\le\beta$.
\end{proof}

 The following category-theoretic result is easy, but one has to be
careful.

\begin{lem} \label{cartesian-product-lemma}
 Let $\kappa$~be a regular cardinal and $(\sA_\xi)_{\xi\in\Xi}$ be
a family of $\kappa$\+accessible categories, indexed by a set\/ $\Xi$
of the cardinality smaller than~$\kappa$.
 Then the Cartesian product\/ $\sA=\prod_{\xi\in\Xi}\sA_\xi$ of
the categories\/ $\sA_\xi$ is also a $\kappa$\+accessible category,
and the $\kappa$\+presentable objects of\/ $\sA$ are precisely all
the collections of objects $(A_\xi\in\sA_\xi)_{\xi\in\Xi}$ with
$A_\xi\in(\sA_\xi)_{<\kappa}$ for every\/ $\xi\in\Xi$.
\end{lem}

\begin{proof}
 This assertion appears in~\cite[Proposition~2.67]{AR}, but with
the condition on the cardinality of $\Xi$ missing.
 For details, see~\cite[Proposition~2.1]{Pacc}.
\end{proof}

 Let $\sA$ and $\sB$ be two categories, and let $P\:\sA\rarrow\sB$
and $Q\:\sA\rarrow\sB$ be a pair of parallel functors.
 The \emph{isomorpher}~\cite[Remark~5.2]{Pacc} of the two functors
$P$ and $Q$ is defined as the category $\sD$ of all pairs $(A,\theta)$,
where $A\in\sA$ is an object and $\theta\:P(A)\simeq Q(A)$ is
an isomorphism in~$\sB$.

\begin{prop} \label{isomorpher-prop}
 Let $\kappa$~be a regular cardinal, and $\lambda<\kappa$ be a smaller
infinite cardinal.
 Let\/ $\sA$ and\/ $\sB$ be $\kappa$\+accessible categories in
which the colimits of all $\lambda$\+indexed chains exist.
 Let $P\:\sA\rarrow\sB$ and $Q\:\sA\rarrow\sB$ be two functors
preserving $\kappa$\+directed colimits and colimits of
$\lambda$\+indexed chains, and taking $\kappa$\+presentable objects to
$\kappa$\+presentable objects.
 Then the isomorpher\/ $\sD$ of the pair of functors $P$ and $Q$ is
a $\kappa$\+accessible category, and the $\kappa$\+presentable objects
of\/ $\sD$ are precisely all the objects $(A,\theta)\in\sD$ with
$A\in\sA_{<\kappa}$.
\end{prop}

\begin{proof}
 This is a corollary of Proposition~\ref{pseudopullback-prop}, as
explained in~\cite[Remark~5.2]{Pacc}.
\end{proof}

 Let $\sK$ and $\sL$ be two categories, let $F$, $G\:\sK
\rightrightarrows\sL$ be a pair of parallel functors, and let
$\phi$, $\psi\:F\rightrightarrows G$ be a pair of parallel natural 
transformations.
 The \emph{equifier}~\cite[Lemma~2.76]{AR} of the pair of natural
transformations $\phi$ and~$\psi$ is the full subcategory
$\sE\subset\sK$ consisting of all the objects $E\in\sK$ for which
the two morphisms $\phi_E\:F(E)\rarrow G(E)$ and $\psi_E\:F(E)
\rarrow G(E)$ in $\sL$ are equal to each other, that is
$\phi_E=\psi_E$.

\begin{prop} \label{equifier-prop}
 Let $\kappa$~be a regular cardinal, and $\lambda<\kappa$ be a smaller
infinite cardinal.
 Let\/ $\sK$ and\/ $\sL$ be $\kappa$\+accessible categories in
which the colimits of all $\lambda$\+indexed chains exist.
 Let $F\:\sK\rarrow\sL$ and $G\:\sK\rarrow\sL$ be two functors
preserving $\kappa$\+directed colimits and colimits of
$\lambda$\+indexed chains; assume further that the functor $F$ takes
$\kappa$\+presentable objects to $\kappa$\+presentable objects.
 Let $\phi\:F\rarrow G$ and $\psi\:F\rarrow G$ be two natural
transformations.
 Then the equifier\/ $\sE$ of the pair of natural transformations $\phi$
and~$\psi$ is a $\kappa$\+accessible category, and
the $\kappa$\+presentable objects of\/ $\sE$ are precisely all
the objects of\/ $\sE$ that are $\kappa$\+presentable in\/~$\sK$.
\end{prop}

\begin{proof}
 This is~\cite[Theorem~3.1]{Pacc}.
 The result goes back to the unpublished preprint~\cite[Theorem~3.8,
Corollary~3.9, and Remark~3.11(II)]{Ulm}.
\end{proof}

 Let $X$ be a scheme.
 We will say that $X$ is \emph{countably quasi-compact} if every open
covering of $X$ contains an at most countable subcovering.
 Equivalently, a scheme is countably quasi-compact if and only if it
can be covered by (at most) countably many affine open subschemes.

 Let $f\:Y\rarrow X$ be a morphism of schemes.
 We will say that the morphism~$f$ is \emph{countably quasi-compact}
if, for every countably quasi-compact open subscheme $U\subset X$,
the preimage $f^{-1}(U)\subset Y$ is a countably quasi-compact scheme.
 Equivalently, $f$~is countably quasi-compact if and only if, for
every affine open subscheme $U\subset X$, the scheme $f^{-1}(U)$ is
countably quasi-compact.
 It suffices to check this condition for open subschemes $U_\alpha
\subset X$ belonging to any fixed affine open covering of~$X$.

 We will say that a scheme $X$ is \emph{countably quasi-separated} if
the diagonal morphism to the Cartesian product $X\rarrow
X\times_{\Spec\boZ} X$ is countably quasi-compact.
 Equivalently, $X$ is countably quasi-separated if and only if
the intersection of any two countably quasi-compact open subschemes
in $X$ is countably quasi-compact, and if and only if the intersection
of any two affine open subschemes in $X$ is countably quasi-compact.

 The following theorem is the full version of the main result of
this paper.

\begin{thm} \label{flat-qcoh-on-countably-qcomp-qsep}
 Let $X$ be a countably quasi-compact, countably quasi-separated scheme.
 Then the category of flat quasi-coherent sheaves $X\qcoh_\fl$ is\/
$\aleph_1$\+accessible.
 The $\aleph_1$\+presentable objects of $X\qcoh_\fl$ are precisely all
the locally countably presentable flat quasi-coherent sheaves on~$X$.
 So every flat quasi-coherent sheaf on $X$ is
an\/ $\aleph_1$\+directed colimit of locally countably presentable
flat quasi-coherent sheaves.
\end{thm}

\begin{proof}
 An easy inductive argument produces a countable affine open covering
$X=\bigcup_{\alpha\in\Delta}U_\alpha$ of the scheme $X$ such that
$U_\alpha\cap U_\beta=\bigcup_{\gamma\le\alpha}^{\gamma\le\beta}
U_\gamma$ for all $\alpha$, $\beta\in\Delta$.
 Then Lemma~\ref{glue-from-affine-covering} is applicable.
 We are interested in flat quasi-coherent sheaves; so we will
consider sets of data $(M_\alpha,m_\beta^\alpha)$ as in the lemma
with flat $R_\alpha$\+modules~$M_\alpha$.

 Let us first deal with the category of sets of data~(1\<2),
disregarding condition~(3).
 Consider the categories $\sA=\prod_{\alpha\in\Delta}R_\alpha\Modl_\fl$
and $\sB=\prod_{\beta\le\alpha\in\Delta} R_\beta\Modl_\fl$.
 Proposition~\ref{flat-modules-accessible} together with
Lemma~\ref{cartesian-product-lemma} (for $\kappa=\aleph_1$) tell us that
both the categories $\sA$ and $\sB$ are $\aleph_1$\+accessible, and
describe their full subcategories of $\aleph_1$\+presentable objects.

 Consider the following pair of parallel functors $F$, $G\:\sA
\rightrightarrows\sB$.
 To a collection of flat modules
$(M_\alpha\in R_\alpha\Modl_\fl)_{\alpha\in\Delta}\in\sA$,
the functor $F$ assigns the collection of flat modules
$(L_{\alpha,\beta}\in R_\beta\Modl_\fl)_{\beta\le\alpha\in\Delta}\in\sB$
with $L_{\alpha,\beta}=R_\beta\ot_{R_\alpha}M_\alpha$.
 To the same collection of flat modules
$(M_\alpha\in R_\alpha\Modl_\fl)_{\alpha\in\Delta}\in\sA$,
the functor $G$ assigns the collection of flat modules
$(N_{\alpha,\beta}\in R_\beta\Modl_\fl)_{\beta\le\alpha\in\Delta}\in\sB$
with $N_{\alpha,\beta}=M_\beta$.

 Then the isomorpher $\sD$ of the pair of functors $F$ and $G$ is
the desired category of sets of data~(1\<2).
 Proposition~\ref{isomorpher-prop} (for $\kappa=\aleph_1$ and
$\lambda=\aleph_0$) tells us that the category $\sD$ is
$\aleph_1$\+accessible, and provides a description of its full
subcategory of $\aleph_1$\+presentable objects.

 In order to impose condition~(3), we need to apply the construction
of the equifier.
 Put $\sK=\sD$ and $\sL=
\prod_{\gamma\le\beta\le\alpha\in\Delta}R_\gamma\Modl_\fl$.
 Proposition~\ref{flat-modules-accessible} together with
Lemma~\ref{cartesian-product-lemma} (for $\kappa=\aleph_1$) tell us that
the category $\sL$ is $\aleph_1$\+accessible, and describe its full
subcategory of $\aleph_1$\+presentable objects.

 Consider the following pair of parallel functors $F$, $G\:\sK
\rightrightarrows\sL$.
 To a set of data $(M_\alpha,m_\beta^\alpha)\in\sK=\sD$
the functor $F$ assigns the collection of flat modules
$(L_{\alpha,\beta,\gamma}\in R_\gamma\Modl_\fl)_
{\gamma\le\beta\le\alpha\in\Delta}\in\sL$
with $L_{\alpha,\beta,\gamma}=R_\gamma\ot_{R_\alpha}M_\alpha$.
 To the same set of data $(M_\alpha,m_\beta^\alpha)\in\sK$,
the functor $G$ assigns the collection of flat modules
$(N_{\alpha,\beta,\gamma}\in R_\gamma\Modl_\fl)_
{\gamma\le\beta\le\alpha\in\Delta}\in\sL$
with $N_{\alpha,\beta,\gamma}=M_\gamma$.

 Furthermore, consider the following pair of parallel natural
transformations $\phi$, $\psi\:F\rightrightarrows G$.
 To a set of data $(M_\alpha,m_\beta^\alpha)\in\sK=\sD$, the natural
transformation~$\phi$ assigns the collection of morphisms of flat
modules $m_\gamma^\alpha\:L_{\alpha,\beta,\gamma}=
R_\gamma\ot_{R_\alpha}M_\alpha\rarrow M_\gamma=N_{\alpha,\beta,\gamma}$.
 To the same set of data $(M_\alpha,m_\beta^\alpha)\in\sK$,
the natural transformation~$\psi$ assigns the collection of
compositions of morphisms of flat modules
$$
 \xymatrix{
  {L_{\alpha,\beta,\gamma}=R_\gamma\ot_{R_\alpha}M_\alpha}
  \ar[rr]^-{R_\gamma\ot_{R_\beta}m_\beta^\alpha}
  && {R_\gamma\ot_{R_\beta}M_\beta} \ar[rr]^-{m_\gamma^\beta}
  && {M_\gamma=N_{\alpha,\beta,\gamma}.}
 }
$$

 Then the equifier $\sE$ of the pair of natural transformations $\phi$
and~$\psi$ is the category of sets of data~(1\<2) satisfying
condition~(3), i.~e., $\sE\simeq X\qcoh_\fl$.
 Proposition~\ref{equifier-prop} (for $\kappa=\aleph_1$ and
$\lambda=\aleph_0$) tells us that the category $\sE$ is
$\aleph_1$\+accessible, and provides the desired description of
its full subcategory of $\aleph_1$\+presentable objects.
\end{proof}

\begin{rem}
 Let $\kappa$~be a regular cardinal.
 Similarly to the definitions above, one defines
\emph{$\kappa$\+quasi-compact} and \emph{$\kappa$\+quasi-separated}
schemes (so that these properties for $\kappa=\aleph_0$ reduce to
the usual quasi-compactness and quasi-separatedness, while
$\kappa=\aleph_1$ gives the countable versions).
 Theorem~\ref{flat-qcoh-on-countably-qcomp-qsep} generalizes to
the claims that, for an uncountable regular cardinal~$\kappa$ and
a $\kappa$\+quasi-compact, $\kappa$\+quasi-separated scheme $X$,
the category $X\qcoh_\fl$ is $\kappa$\+accessible, and its
$\kappa$\+presentable objects are precisely all the locally
$\kappa$\+presentable flat quasi-coherent sheaves.

 Furthermore, one can consider the abelian category $X\qcoh$ instead of
the additive category $X\qcoh_\fl$.
 Then it is long known that, for a quasi-compact quasi-separated scheme
$X$, the category $X\qcoh$ is locally finitely presentable, and its
finitely presentable objects are precisely all the locally finitely
presentable quasi-coherent sheaves~\cite[0.5.2.5 and
Corollaire~I.6.9.12]{EGA1}, \cite[Definition Tag~01BN and
Lemma Tag~01PJ or~01PK]{SP}.
 For an uncountable regular cardinal~$\kappa$, the result of
Gabber~\cite[Lemma Tag~077N]{SP} essentially tells that, on
a $\kappa$\+quasi-compact, $\kappa$\+quasi-separated scheme $X$,
any quasi-coherent sheaf is a $\kappa$\+directed union of its
locally $\kappa$\+generated quasi-coherent subsheaves.
 This can be restated by saying that the category $X\qcoh$ is
``locally $\kappa$\+generated'' in the sense that it safisfies
the assumptions of~\cite[Local Generation Theorem~1.70]{AR}
for the cardinal~$\kappa$.

 Gabber's theorem can be strengthened using the techniques of
the proof of Theorem~\ref{flat-qcoh-on-countably-qcomp-qsep} above.
 In fact, for any regular cardinal~$\kappa$ and any
$\kappa$\+quasi-compact, $\kappa$\+quasi-separated scheme $X$,
the abelian category $X\qcoh$ is localy $\kappa$\+presentable.
 All colimits exist in $X\qcoh$, so it suffices to show that
$X\qcoh$ is $\kappa$\+accessible.
 The case of $\kappa=\aleph_0$ is due to Grothendieck and covered by
the references to~\cite{EGA1} and~\cite{SP} above.
 In the case when $\kappa$~is uncountable, one argues as in the proof
of Theorem~\ref{flat-qcoh-on-countably-qcomp-qsep}, replacing
the reference to Proposition~\ref{flat-modules-accessible} by
Lemma~\ref{modules-locally-presentable} and all mentions of
the categories of flat modules $R\Modl_\fl$ by the categories of
all modules $R\Modl$.
 Otherwise, the argument is the same.
 (Cf.~\cite[Remarks~3.2 and~3.4]{Pflcc}.)
\end{rem}

\Section{Complexes of Flat Quasi-Coherent Sheaves as Directed Colimits}
\label{complexes-of-flats-secn}

 In this section we prove three theorems about complexes of
quasi-coherent sheaves, based on the respective module versions
provided by Propositions~\ref{complexes-of-flats-acessible},
\ref{acyclic-flat-cocycles-accessible},
and~\ref{homotopy-flats-accessible}.

\begin{thm}
 Let $X$ be a countably quasi-compact, countably quasi-separated
scheme.
 Then the category of complexes of flat quasi-coherent sheaves\/
$\Com(X\qcoh_\fl)$ is\/ $\aleph_1$\+accessible.
 The\/ $\aleph_1$\+presentable objects of\/ $\Com(X\qcoh_\fl)$ are
precisely all the complexes of locally countably presentable flat
quasi-coherent sheaves on~$X$.
 Hence every complex of flat quasi-coherent sheaves on $X$ is
an\/ $\aleph_1$\+directed colimit of complexes of locally countably
presentable flat quasi-coherent sheaves.
\end{thm}

\begin{proof}
 The proof is similar to that of
Theorem~\ref{flat-qcoh-on-countably-qcomp-qsep} and based on
Proposition~\ref{complexes-of-flats-acessible} (instead of
Proposition~\ref{flat-modules-accessible}).
 One also needs to use the (obvious) version of
Lemma~\ref{glue-from-affine-covering} for complexes of
quasi-coherent sheaves.

 In the case of a quasi-compact quasi-separated scheme $X$,
an argument similar to the proof of
Theorem~\ref{flat-qcoh-on-qcomp-qsep} also works.
 Then one needs to use Proposition~\ref{complexes-of-flats-acessible}
and the version of Lemma~\ref{glue-from-a-cover-by-two} for
complexes of quasi-coherent sheaves.
\end{proof}

\begin{thm}
 Let $X$ be a countably quasi-compact, countably quasi-separated
scheme.
 Then the category\/ $\Ac_\fl(X\qcoh_\fl)$ of acyclic complexes of
flat quasi-coherent sheaves on $X$ with flat sheaves of cocycles
is\/ $\aleph_1$\+accessible.
 The\/ $\aleph_1$\+presentable objects of\/ $\Ac_\fl(X\qcoh_\fl)$ are
precisely all the acyclic complexes of locally countably presentable
flat quasi-coherent sheaves on $X$ with flat sheaves of cocycles.
 So every acyclic complex of flat quasi-coherent sheaves on $X$
with flat sheaves of cocycles is an\/ $\aleph_1$\+directed colimit of
acyclic complexes of locally countably presentable flat quasi-coherent
sheaves with flat sheaves of cocycles.
\end{thm}

\begin{proof}
 The proof is similar to that of
Theorem~\ref{flat-qcoh-on-countably-qcomp-qsep} and based on
Proposition~\ref{acyclic-flat-cocycles-accessible} (instead of
Proposition~\ref{flat-modules-accessible}).
 In the case of a quasi-compact quasi-separated scheme $X$,
an argument similar to the proof of
Theorem~\ref{flat-qcoh-on-qcomp-qsep} is also possible.

 For both versions of the argument, one needs to make the obvious 
observation that a complex of flat quasi-coherent sheaves $\F^\bu$
on $X$ has flat sheaves of cocycles if and only if, for every affine
open subscheme $U\subset X$, the complex of flat $\O(U)$\+modules
$\F^\bu(U)$ has flat $\O(U)$\+modules of cocycles.
 Moreover, it suffices to check this condition for affine open
subschemes $U_\alpha\subset X$ belonging to any given affine open
covering of the scheme~$X$.
 So being a complex of flat quasi-coherent sheaves with flat sheaves
of cocycles is a local property.
\end{proof}

 In order to obtain the quasi-coherent sheaf version of
the Christensen--Holm theorem~\cite[Theorem~1.1]{CH}, based on
Proposition~\ref{homotopy-flats-accessible}, we will need 
a preparatory lemma.
 The resulting theorem will require stronger assumptions on the scheme
$X$ than the previous theorems, because of the assumption in
part~(b) of the lemma.

 But first of all, let us give the definition.
 A complex of flat quasi-coherent sheaves $\F^\bu$ on a scheme $X$ is
said to be \emph{homotopy flat} if, for every acyclic complex of
quasi-coherent sheaves $\A^\bu$ on $X$, the complex of quasi-coherent
sheaves $\A^\bu\ot_{\O_X}\F^\bu$ on $X$ is also acyclic.

\begin{lem} \label{homotopy-flatness-locality-lemma}
\textup{(a)} If $\F^\bu$ is complex of quasi-coherent sheaves on
a scheme $X$ and $X=\bigcup_{\alpha\in\Delta} U_\alpha$ is an open
covering of $X$, and if the quasi-coherent sheaves $\F^\bu|_{U_\alpha}$
are homotopy flat on $U_\alpha$ for all\/ $\alpha\in\Delta$, then
the quasi-coherent sheaf $\F$ on $X$ is homotopy flat. \par
\textup{(b)} If $\F^\bu$ is a homotopy flat complex of quasi-coherent
sheaves on a scheme $X$ and $U\subset X$ is an open subscheme such that
the open immersion $j\:U\rarrow X$ is an affine morphism, then
$\F^\bu|_U$ is a homotopy flat complex of flat quasi-coherent sheaves
on~$U$.
\end{lem}

\begin{proof}
 Part~(a) is easy.
 Let $\A^\bu$ be an acyclic complex of quasi-coherent sheaves on~$X$.
 Then $\A^\bu|_{U_\alpha}$ is an acyclic complex of quasi-coherent
sheaves on $U_\alpha$ for every $\alpha\in\Delta$.
 By assumption, it follows that the complex of quasi-coherent sheaves
$\A^\bu|_{U_\alpha}\ot_{\O_{U_\alpha}}\F^\bu|_{U_\alpha}$ is acyclic on
$U_\alpha$.
 Since $(\A^\bu\ot_{\O_X}\F^\bu)|_{U_\alpha}=
\A^\bu|_{U_\alpha}\ot_{\O_{U_\alpha}}\F^\bu|_{U_\alpha}$ and acyclicity
of a complex of quasi-coherent sheaves is a local property, we can
conclude that the complex of quasi-coherent sheaves
$\A^\bu\ot_{\O_X}\F^\bu$ is acyclic on~$X$.

 Part~(b): let $\B^\bu$ be an acyclic complex of quasi-coherent sheaves
on~$U$.
 Since the direct image functor $j_*\:U\qcoh\rarrow X\qcoh$ is exact,
the complex of quasi-coherent sheaves $j_*\B^\bu$ on~$X$ is acyclic
as well.
 By assumption, it follows that the complex of quasi-coherent sheaves
$j_*\B^\bu\ot_{\O_X}\F^\bu$ is acyclic on~$X$.
 Hence the complex of quasi-coherent sheaves
$\B^\bu\ot_{\O_U}(\F^\bu|_U)\simeq(j_*\B^\bu\ot_{\O_X}\F^\bu)|_U$
is acyclic on~$U$.
\end{proof}

\begin{cor} \label{homotopy-flatness-locality-cor}
 Let $X$ be a semi-separated scheme and $X=\bigcup_{\alpha\in\Delta}
U_\alpha$ be its affine open covering.
 Then a complex of quasi-coherent sheaves $\F^\bu$ on $X$ is homotopy
flat if and only if, for every $\alpha\in\Delta$, the complex of
$\O(U_\alpha)$\+modules $\F^\bu(U_\alpha)$ is homotopy flat.
\end{cor}

\begin{proof}
 Semi-separated schemes $X$ are characterized by the condition that,
for any affine open subscheme $U\subset X$, the open immersion
$j\:U\rarrow X$ is an affine morphism.
 So both parts of Lemma~\ref{homotopy-flatness-locality-lemma}
are applicable.
 (Cf.\ Remark~\ref{prolongation-problem-remark}.)
\end{proof}

\begin{thm}
 Let $X$ be a countably quasi-compact, semi-separated scheme.
 Then the category\/ $\Com(X\qcoh_\fl)_\hfl$ of homotopy flat
complexes of flat quasi-coherent sheaves on $X$ is\/
$\aleph_1$\+accessible.
 The\/ $\aleph_1$\+presentable objects of\/ $\Com(X\qcoh_\fl)_\hfl$
are precisely all the homotopy flat complexes of locally countably
presentable flat quasi-coherent sheaves on~$X$.
 So every homotopy flat complex of flat quasi-coherent sheaves on $X$
is an\/ $\aleph_1$\+directed colimit of homotopy flat complexes of
locally countably presentable flat quasi-coherent sheaves.
\end{thm}

\begin{proof}
 The proof is similar to that of
Theorem~\ref{flat-qcoh-on-countably-qcomp-qsep} and based on
Proposition~\ref{homotopy-flats-accessible} for $\kappa=\aleph_1$
(instead of Proposition~\ref{flat-modules-accessible}).
 In the case of a quasi-compact semi-separated scheme $X$,
an argument similar to the proof of
Theorem~\ref{flat-qcoh-on-qcomp-qsep} also works.
 Corollary~\ref{homotopy-flatness-locality-cor} is important for
both versions of the argument.
\end{proof}

\Section{Quasi-Coherent Sheaves of Finite Projective Dimension}
\label{qcoh-finite-projdim-secn}

 Much of the rest of this paper is written in the setting of
exact categories (in Quillen's sense).
 We refer to the overview~\cite{Bueh} for background material on exact
categories.

 Let $\sK$ be an exact category, $P\in\sK$ be an object, $m\ge0$ be
an integer.
 One says that the \emph{projective dimension} of $P$ in $\sK$ does not
exceed~$m$ and writes $\pd_\sK P\le m$ if $\Ext_\sK^n(P,K)=0$ for
all objects $K\in\sK$ and integers $n>m$.
 Here $\Ext_\sK^*$ denotes the Yoneda Ext groups in the exact
category~$\sK$.

 For a scheme $X$, we denote by $\Ext_X^*=\Ext_{X\qcoh}^*$
the Yoneda Ext groups in the abelian category $X\qcoh$.
 Given a quasi-coherent sheaf $\P$ on a scheme $X$, we denote by
$\pd_X\P=\pd_{X\qcoh}\P$ the projective dimension of the object
$\P\in X\qcoh$.
 Notice that the definition of the projective dimension in the previous
paragraph does \emph{not} require any projective resolutions in
$X\qcoh$ (which usually do not exist).

 The following Ext-adjunction lemma is standard.

\begin{lem} \label{ext-adjunction}
 Let\/ $\sA$ and\/ $\sB$ be exact categories, $F\:\sA\rarrow\sB$ be
a functor, and $G\:\sB\rarrow\sA$ be a functor right adjoint to~$F$.
 Assume that both the functors $F$ and $G$ are exact.
 Then for any two objects $A\in\sA$ and $B\in\sB$, and every integer
$n\ge0$, there is a natural isomorphism of the Ext groups
$$
 \Ext_\sB^n(F(A),B)\simeq\Ext_\sA^n(A,G(B)).
$$
\end{lem}

\begin{proof}
 One simple approach consists in establishing a more general result,
viz., that an adjunction of exact functors $F$ and $G$ leads to
an adjunction of the induced triangulated functors between the (bounded
or unbounded) derived categories $F\:\sD(\sA)\rarrow\sD(\sB)$ and
$G\:\sD(\sB)\rarrow\sD(\sA)$.
 For this purpose, one constructs the adjunction morphisms for the pair
of triangulated functors $F$ and $G$ and checks the required equations
on the compositions.
\end{proof}

 The adjoint pairs of exact functors we are interested in in this
section are the inverse and direct images for affine open immersions
of affine open subschemes.
 Given a semi-separated scheme $X$ and an affine open subscheme
$U\subset X$ with the identity open immersion denoted by
$j\:U\rarrow X$, we have an exact functor of direct image
$j_*\:U\qcoh\rarrow X\qcoh$ and an exact functor of inverse image
$j^*\:X\qcoh\rarrow U\qcoh$.
 The latter functor is left adjoint to the former one.
 
 The following well-known lemma provides \v Cech coresolutions in
the categories of quasi-coherent sheaves on a quasi-compact
semi-separated scheme.

\begin{lem} \label{cech-coresolution}
 Let $X=\bigcup_{\alpha=1}^d U_\alpha$ be a finite afffine open covering
of a quasi-compact semi-separated scheme~$X$.
 For any sequence of indices $1\le\alpha_1<\alpha_2<\dotsb<
\alpha_r\le d$, denote by
$j_{\alpha_1,\dotsc,\alpha_r}\:\bigcap_{s=1}^r U_{\alpha_s}\rarrow X$
the identity open immersion of the corresponding (affine) intersection of
the affine open subschemes $U_{\alpha_1}\cap\dotsb\cap U_{\alpha_r}
\subset X$.
 Then, for any quasi-coherent sheaf $\M$ on $X$, there is a natural
$(d+\nobreak1)$\+term exact sequence of quasi-coherent sheaves on $X$
\begin{multline} \label{cech-coresolution-eqn}
 0\lrarrow \M\lrarrow
 \bigoplus\nolimits_{\alpha=1}^d j_\alpha{}_*j_\alpha^*\M
 \lrarrow\bigoplus\nolimits_{1\le\alpha<\beta\le d}
 j_{\alpha,\beta}{}_*j_{\alpha,\beta}^*\M \\
 \lrarrow\dotsb\lrarrow j_{1,2,\dotsc,d}{}_*j_{1,2,\dotsc,d}^*\M
 \lrarrow0.
\end{multline}
\end{lem}

\begin{proof}
 This is standard; see~\cite[Lemma~III.4.2]{Har}.
\end{proof}

 The following theorem is the main result of this section.

\begin{thm} \label{qcoh-projdim-theorem}
 Let $X=\bigcup_{\alpha=1}^d U_\alpha$ be a finite affine open covering
of a quasi-compact semi-separated scheme $X$, and let $\P$ be
a quasi-coherent sheaf on~$X$.
 Then the projective dimension of $\P$ on $X$ is finite if and only if
the projective dimension of $j_\alpha^*\P$ on $U_\alpha$ is finite
for every\/ $1\le\alpha\le d$.
 More precisely: \par
\textup{(a)} if\/ $\pd_X\P\le m$, then\/ $\pd_U j^*\P\le m$ for
any affine open subscheme $U\subset X$ with the identity open
immersion morphism $j\:U\rarrow X$; \par
\textup{(b)} if\/ $\pd_{U_\alpha}j_\alpha^*\P\le m$ for all\/
$1\le\alpha\le d$, then\/ $\pd_X\P\le m+d-1$.
\end{thm}

\begin{proof}
 Part~(a) follows from the isomorphisms $\Ext_U^n(j^*\P,\M)\simeq
\Ext_X^n(\P,j_*\M)$ provided by Lemma~\ref{ext-adjunction} for
all quasi-coherent sheaves $\M$ on $U$ and integers $n\ge0$.

 To prove part~(b), consider the \v Cech
coresolution~\eqref{cech-coresolution-eqn} of a quasi-coherent
sheaf $\M$ on $X$ and use it to describe the groups $\Ext_X^n(\P,\M)$.
 By part~(a), the projective dimension of the quasi-coherent sheaf
$j_{\alpha_1,\dotsc,\alpha_r}^*\P$ on $\bigcap_{s=1}^r U_{\alpha_s}$,
where $1\le r\le d$, does not exceed that of
$j_{\alpha_1}^*\P$.
 The isomorphism provided by Lemma~\ref{ext-adjunction} then tells us that 
\[ \Ext_X^n(\P,j_{\alpha_1,\dotsc,\alpha_r}{}_*j_{\alpha_1,\dotsc,\alpha_r}^*\M)\simeq\Ext_{\bigcap_{s=1}^r U_{\alpha_s}}^n(j_{\alpha_1,\dotsc,\alpha_r}^*\P,j_{\alpha_1,\dotsc,\alpha_r}^*\M), \]
which vanishes for $n>m$.
Using the exactness of~\eqref{cech-coresolution-eqn}, a straightforward
induction shows that $\Ext_X^n(\P,\M)=0$ whenever $n\ge m+d$.
\end{proof}

\begin{rem}
 One can notice that the proof of
Theorem~\ref{qcoh-projdim-theorem}(a) does not need the open subscheme
$U\subset X$ to be affine.
 Rather, it is the open immersion \emph{morphism} $j\:U\rarrow X$ that
has to be affine for the argument to work (cf.\ the assumptions of
Lemma~\ref{homotopy-flatness-locality-lemma}(b)).
 On the other hand, for \emph{nonaffine} open immersion morphisms~$j$,
the assertion of Theorem~\ref{qcoh-projdim-theorem}(a) is certainly
\emph{not} true.
 For a counterexample, it suffices to take $X=\Spec k[x,y]$ to be
the affine plane over a field~$k$, and $U\subset X$ to be
the complement to a closed point, $U=X\setminus\{\mathfrak m\}$,
where $\mathfrak m=(x,y)\subset k[x,y]$.
 Then the structure sheaf $\O_X$ is a projective object in
the category $X\qcoh$; indeed, $\O_X$ corresponds to the free
$k[x,y]$\+module with one generator under the equivalence of $X\qcoh$
with the module category $k[x,y]\Modl$.
 But the object $\O_U=j^*\O_X$ is \emph{not} projective in $U\qcoh$;
in fact, $\Ext^1_U(\O_U,\O_U)\ne0$.
\end{rem}

\begin{lem} \label{countable-govorov-lazard}
 Any countably presentable flat module over an associative ring $R$
can be presented as a countable directed colimit of finitely generated
projective $R$\+modules.
 Consequently, the projective dimension of any countably presentable
flat module does not exceed\/~$1$.
\end{lem}

\begin{proof}
 This is~\cite[Corollary~2.23]{GT}.
\end{proof}

\begin{cor} \label{direct-limits-of-flat-qcoh-of-finite-projdim}
 Let $X=\bigcup_{\alpha=1}^d U_\alpha$ be a finite affine open covering
of a quasi-compact semi-separated scheme~$X$.
 Then the projective dimension of any locally countably presentable
flat quasi-coherent sheaf on $X$ does not exceed~$d$.
 Therefore, any flat quasi-coherent sheaf on $X$ is
an\/ $\aleph_1$\+directed colimit of flat quasi-coherent sheaves
of projective dimension~$\le d$.
\end{cor}

\begin{proof}
 The first assertion easily follows from
Theorem~\ref{qcoh-projdim-theorem}
and Lemma~\ref{countable-govorov-lazard}.
 The second assertion is, in view of the first assertion,
a corollary of Theorem~\ref{flat-qcoh-on-qcomp-qsep}
(which is one of our main theorems).
\end{proof}

\begin{rem} \label{enough-countably-flats-remark}
 By~\cite[Section~2.4]{M-n} or~\cite[Lemma~A.1]{EP}, every
quasi-coherent sheaf $\M$ on a quasi-compact semi-separated scheme is
a quotient sheaf of a flat quasi-coherent sheaf~$\F$.
 The stronger result of~\cite[Lemma~4.1.1]{Pcosh} tells that one can
choose $\F$ to be a \emph{very flat} quasi-coherent sheaf; then,
for any affine open subscheme $U\subset X$, the $\O(U)$\+module
$\F(U)$ is flat of projective dimension at most~$1$.

 Theorem~\ref{flat-qcoh-on-qcomp-qsep} together with
Lemma~\ref{countable-govorov-lazard} allow us to arrive to the same
conclusion in a different way.
 Let $\M$ be a quasi-coherent sheaf on~$X$.
 By~\cite[Section~2.4]{M-n} or~\cite[Lemma~A.1]{EP}, \,$\M$ is
a quotient sheaf of a flat quasi-coherent sheaf~$\G$.
 By Theorem~\ref{flat-qcoh-on-qcomp-qsep}, \,$\G$ is a directed colimit
of some locally countably presentable flat quasi-coherent sheaves
$(\F_\xi)_{\xi\in\Xi}$.
 By Lemma~\ref{countable-govorov-lazard}, the $\O(U)$\+module
$\F_\xi(U)$ is flat of projective dimension at most~$1$ for any
affine open subscheme $U\subset X$ and every $\xi\in\Xi$.
 It remains to observe that $\G$, and consequently also $\M$, is
a quotient sheaf of the quasi-coherent sheaf
$\bigoplus_{\xi\in\Xi}\F_\xi$ whose module of sections over
any affine open subscheme $U\subset X$ also has projective
dimension at most~$1$.
\end{rem}

\Section{Preliminaries on Ext-Orthogonal Classes}
\label{prelim-ext-orthogonal-secn}

 Let $\sK$ be an exact category.
 We are interested in $\Ext^1$\+orthogonal and
$\Ext^{\ge1}$\+or\-thogonal pairs of classes of objects in~$\sK$.
 The classical theory of such Ext-orthogonal classes involves
the concepts of a \emph{cotorsion pair} and particularly
a \emph{complete} cotorsion pair.
 These notions include conditions that are not relevant for the purposes
of the present paper, so we discuss a more general setting.

 Let $\sA$ and $\sB\subset\sK$ be two classes of objects.
 The notation $\sA^{\perp_1}\subset\sK$ stands for the class of
all objects $X\in\sK$ such that $\Ext^1_\sK(A,X)=0$ for all $A\in\sA$.
 Dually, ${}^{\perp_1}\sB\subset\sK$ is the class of all objects
$Y\in\sK$ such that $\Ext^1_\sK(Y,B)=0$ for all $B\in\sB$.

 Furthermore, $\sA^{\perp_{\ge1}}\subset\sK$ denotes the class of
all objects $X\in\sK$ such that $\Ext^n_\sK(A,X)=0$ for all $A\in\sA$
and $n\ge1$.
 Dually, ${}^{\perp_{\ge1}}\sB\subset\sK$ is the class of all objects
$Y\in\sK$ such that $\Ext^n_\sK(Y,B)=0$ for all $B\in\sB$ and $n\ge 1$.

 A class of objects $\sA\subset\sK$ is said to be \emph{generating}
if every object of $\sK$ is an admissible quotient object of
(i.~e., the codomain of an admissible epimorphism acting from)
an object from~$\sA$.
 Dually, a class of objects $\sB\subset\sK$ is said to be
\emph{cogenerating} if every object of $\sK$ is an admissible
subobject of (i.~e., the domain of an admissible monomorphism acting
into) an object from~$\sB$.

 More generally, a class of objects $\sA\subset\sK$ is said to be
\emph{self-generating} if for any admissible epimorphism $K\rarrow A$
in $\sK$ with $A\in\sA$ there exists a morphism $A'\rarrow K$ in $\sK$
with $A'\in\sA$ such that the composition $A'\rarrow K\rarrow A$ is
an admissible epimorphism in~$\sK$.
 A class of objects $\sB\subset\sK$ is said to be
\emph{self-cogenerating} if for any admissible monomorphism $B\rarrow K$
in $\sK$ with $B\in\sB$ there exists a morphism $K\rarrow B'$ in $\sK$
with $B'\in\sB$ such that the composition $B\rarrow K\rarrow B'$ is
a admissible monomorphism in~$\sK$.
 Clearly, any generating class is self-generating, and any cogenerating
class is self-cogenerating.

\begin{lem} \label{garcia-rozas}
 Let\/ $\sK$ be an exact category and\/ $\sA\subset\sK$ be
a self-generating class of objects closed under the kernels
of admissible epimorphisms.
 Then\/ $\sA^{\perp_{\ge1}}=\sA^{\perp_1}\subset\sK$.

 Dually, if\/ $\sB\subset\sK$ is a self-cogenerating class of objects
closed under the cokernels of admissible monomorphisms, then\/
${}^{\perp_{\ge1}}\sB={}^{\perp_1}\sB\subset\sK$.
\end{lem}

\begin{proof}
 This is a partial generalization of the standard characterization of
hereditary cotorsion pairs in abelian/exact categories (going back
to Garc\'\i a Rozas~\cite[Theorem~1.2.10]{GR}).
 The argument from~\cite[Lemma~6.17]{Sto-ICRA}
or~\cite[Lemma~4.25]{SaoSt} applies.
\end{proof}

 Let $\sB\subset\sK$ be a class of objects.
 An object $M\in\sK$ is said to be \emph{$\sB$\+periodic} if there
exists an admissible short exact sequence $0\rarrow M\rarrow B\rarrow M
\rarrow0$ in $\sK$ with $B\in\sB$.
 The following two lemmas are essentially due to Bazzoni,
Cort\'es-Izurdiaga, and Estrada~\cite{BCE}.

\begin{lem} \label{dimension-shifting}
 Let\/ $\sK$ be an exact category and\/ $\sA$, $\sB\subset\sK$ be
a pair of classes of objects such that\/ $\Ext^n_\sK(A,B)=0$ for all
$A\in\sA$, \,$B\in\sB$, and $n\ge1$.
 Let\/ $M\in\sK$ be a\/ $\sB$\+periodic object.
 Then there is an isomorphism of abelian groups $\Ext^n_\sK(A,M)
\simeq\Ext^1_\sK(A,M)$ for all $A\in\sA$ and $n\ge1$.
\end{lem}

\begin{proof}
 See~\cite[Lemma~4.1]{BCE}.
\end{proof}

 Let $\sA$ be a full subcategory closed under extensions in an exact
category~$\sK$.
 Then the class of all short exact sequences in $\sK$ with the terms
belonging to $\sA$ defines an exact category structure on~$\sA$ (called
the exact category structure \emph{inherited from} the exact
category structure of~$\sK$).
 We will speak about admissible monomorphisms, admissible epimorphisms,
and admissible short exact sequences in $\sA$ presuming the inherited
exact category structure on~$\sA$.

\begin{lem} \label{periodic-2-out-of-3-lemma}
 Let\/ $\sK$ be an exact category and\/ $\sA$, $\sB\subset\sK$ be
a pair of classes of objects such that\/ $\Ext^n_\sK(A,B)=0$ for all
$A\in\sA$, \,$B\in\sB$, and $n\ge1$.
 Assume that the class\/ $\sA$ is closed under extensions in\/ $\sK$,
and let $M\in\sK$ be a\/ $\sB$\+periodic object.
 Then the class of objects\/ $\sA\cap{}^{\perp_1}M\subset\sK$ is
closed under the kernels of admissible epimorphisms, the cokernels of
admissible monomorphisms, and extensions in the exact category\/~$\sA$.
\end{lem}

\begin{proof}
 The argument from~\cite[Lemma~4.4]{BCE} applies.
\end{proof}

Another consequence of Lemma~\ref{dimension-shifting}
is the following interaction of $\sB$\+periodic
objects with objects of finite projective dimension.

\begin{lem} \label{fin-proj-dim-are-ok}
 Let\/ $\sK$ be an exact category and\/ $\sA$, $\sB\subset\sK$ be
a pair of classes of objects such that\/ $\Ext^n_\sK(A,B)=0$ for all
$A\in\sA$, \,$B\in\sB$, and $n\ge1$.
 Let $M\in\sK$ be a\/ $\sB$\+periodic object and $A\in\sA$ be
an object having finite projective dimension in~$\sK$.
 Then\/ $\Ext_\sK^n(A,M)=0$ for all $n\ge1$.
 In particular, the class of objects\/ $\sA\cap{}^{\perp_1}M\subset\sK$
contains all objects of $\sA$ having finite projective dimension
in~$\sK$.
\end{lem}

\begin{proof}
 Follows immediately from Lemma~\ref{dimension-shifting}. 
\end{proof}

 The next definition appeared already in
Section~\ref{flat-sheaves-on-qcomp-qsep-schemes-secn}.
 By a \emph{well-ordered chain} (of objects and morphisms) in
a category $\sK$ one means a directed diagram
$(K_i\to K_j)_{0\le i<j<\alpha}$ indexed by an ordinal~$\alpha$.
 For convenience, we put $K_\alpha=\varinjlim_{i<\alpha}K_i$ (assuming
that the directed colimit exists).
 A well-ordered chain is said to be \emph{smooth} if $K_j=
\varinjlim_{i<j}K_i$ for all limit ordinals $0<j<\alpha$.

 A smooth well-ordered chain $(F_i\to F_j)_{0\le i<j<\alpha}$ in $\sK$
is said to be a \emph{filtration} (of the object $F_\alpha=
\varinjlim_{i<\alpha}F_i$) if $F_0=0$ and the morphisms
$F_i\rarrow F_{i+1}$ are admissible monomorphisms in $\sK$ for all
ordinals $i<i+1<\alpha$.
 In this case, the object $F_\alpha$ is said to be \emph{filtered by}
the cokernels of the morphisms $F_i\rarrow F_{i+1}$,
\,$0\le i<i+1<\alpha$.
 In an alternative terminology, the object $F_\alpha$ is said to be
a \emph{transfinitely iterated extension} of the objects
$\coker(F_i\to\nobreak F_{i+1})$, \,$0\le i<i+1<\alpha$, in this case.

 Notice that we make \emph{no} assumptions of exactness of directed
colimits in $\sK$ (or even \emph{existence} of any other directed 
colimits than those appearing in a particular smooth chain).
 Given a class of objects $\sS\subset\sK$, the class of all objects
filtered by objects from $\sS$ is denoted by $\Fil(\sS)\subset\sK$.

 The following result is known classically as
the \emph{Eklof lemma}~\cite[Lemma~6.2]{GT}.

\begin{lem} \label{eklof-lemma}
 For any exact category\/ $\sK$ and any class of objects\/
$\sB\subset\sK$, the class ${}^{\perp_1}\sB\subset\sK$ is closed under
transfinitely iterated extensions in\/~$\sK$.
 In other words, ${}^{\perp_1}\sB=\Fil({}^{\perp_1}\sB)$.
\end{lem}

\begin{proof}
 The argument from~\cite[Lemma~4.5]{PR} is applicable.
 The generalization from abelian categories (considered in~\cite{PR})
to exact categories is straightforward.
 
 Alternatively, the (somewhat more complicated) argument
from~\cite[Proposition~2.10]{Sto2} (based on~\cite[Lemma~A.3]{Sto2})
is applicable as well.
 Let us point out that the definition of a filtration
in~\cite[Definition~2.9]{Sto2} is more restrictive than our definition
above, in that (in the notation above) all the morphisms
$F_i\rarrow F_j$, \,$0\le i<j\le\alpha$ are assumed to be
admissible monomorphisms in~\cite{Sto2}.
 For counterexamples showing that the two definitions are indeed
different, see~\cite[Examples~4.4]{PR}.
 Still, the argument from~\cite{Sto2} works under our more relaxed
assumptions as well.
\end{proof}

\Section{Exact Categories with Exact Directed Colimits}
\label{exact-categs-exact-dirlims-secn}

 The aim of this section is work out a common generalization
of~\cite[Proposition~7.16]{PS5} and the arguments
in~\cite[Lemma~4.6 and Theorem~4.7]{BCE}.
 Essentially, we replace an abelian category $\sA$ in the setting
of~\cite[Section~7.5]{PS5} by an exact one.
 In our intended application in the next
Sections~\ref{categorical-cotorsion-periodicity-secn}\+-%
\ref{qcoh-cotorsion-periodicity-secn}, \,$\sA$~will be a self-resolving
subcategory (typically, the left class of a hereditary cotorsion pair)
closed under directed colimits in an abelian/exact category~$\sK$.
 Thereby, we obtain a version of the arguments in~\cite[Section~4]{BCE}
\emph{not} using any purity considerations (but only flatness).

 Let $\sA$ be an exact category in which all (set-indexed)
directed colimits exist.
 We will say that $\sA$ \emph{has exact directed colimits} if any
directed colimit of admissible short exact sequences is an admissible
short exact sequence in~$\sA$.
 In this case, all set-indexed coproducts also exist in $\sA$, and
admissible short exact sequences are preserved by coproducts.
 Notice also that any additive category with countable directed
colimits is idempotent-complete.

 The following proposition is formulated and proved in a form
making the similarity with both~\cite[Proposition~7.16]{PS5}
and~\cite[proofs of Lemma~4.6 and Theorem~4.7]{BCE} apparent.

\begin{prop} \label{transfinite-extensions-direct-limits}
 Let\/ $\sA$ be an exact category with exact directed colimits, and
let\/ $\sC\subset\sA$ be a class of objects closed under the cokernels
of admissible monomorphisms and extensions in\/~$\sA$.
 Then the following conditions are equivalent:
\begin{enumerate}
\item $\sC$ is closed under transfinitely iterated extensions
in\/~$\sA$;
\item $\sC$ is closed under the directed colimits of smooth well-ordered
chains of admissible monomorphisms in\/~$\sA$;
\item $\sC$ is closed under the directed colimits of well-ordered chains
of admissible monomorphisms in\/~$\sA$;
\item $\sC$ is closed under the directed colimits of well-ordered chains
in~$\sA$;
\item $\sC$ is closed under directed colimits in\/~$\sA$.
\end{enumerate}
\end{prop}

\begin{proof}
 The implications (5)\,$\Longrightarrow$\,(4)\,%
$\Longrightarrow$\,(3)\,$\Longrightarrow$(2) are obvious.
 The equivalence (1)\,$\Longleftrightarrow$\,(2) is easy (see
the proof in~\cite{PS5} for a discussion).

 (2)\,$\Longrightarrow$\,(3) 
 Let $(C_i\to C_j)_{0\le i<j<\alpha}$ be a well-ordered chain of
admissible monomorphisms in $\sA$ with the objects $C_i\in\sC$ for
all $0\le i<\alpha$.
 We have to prove that $\varinjlim_{i<\alpha}C_\alpha\in\sC$.
 If $\alpha$~is a successor ordinal, then there is nothing to prove.
 Otherwise, let us construct a smooth well-ordered chain
$(D_i\to D_j)_{0\le i<j<\alpha}$ in $\sA$ in the following way:
\begin{itemize}
\item if $j<\omega$, then $D_j=C_j$;
\item if $j>\omega$, \,$j=i+1<\alpha$ is a successor ordinal,
then $D_j=C_i$;
\item if $j<\alpha$~is a limit ordinal, then
$D_j=\varinjlim_{i<j}D_i=\varinjlim_{i<j}C_i$.
\end{itemize}
 The transition morphisms $D_i\rarrow D_j$ for $0\le i<j<\alpha$ are
constructed in the obvious way.
 It is clear that $(D_i\to D_j)_{0\le i<j<\alpha}$ is a smooth chain
in $\sA$ and $\varinjlim_{i<\alpha}D_i=\varinjlim_{i<\alpha}C_i$
(as usual, we denote this directed colimit by $D_\alpha=C_\alpha$).

 The morphisms $D_i\rarrow D_j$ are admissible monomorphisms in $\sA$
for all $0\le i<j<\alpha$, because directed colimits of admissible
monomorphisms are admissible monomorphisms in $\sA$ by assumption.
 To show that $D_j\in\sC$ for all $0\le j<\alpha$, one proceeds by
transfinite induction on~$j$.
 The cases of $j=0$ or $j$~a successor ordinal are obvious, while
the case of a limit ordinal~$j$ is covered by the condition~(2)
(as the objects $(D_i\in\sC)_{0\le i<j}$ form a smooth chain of
admissible monomorphisms).
 Finally, the last application of~(2) shows that $C_\alpha=D_\alpha
\in\sC$, as desired.

 (3)\,$\Longrightarrow$\,(4)
 Let $\alpha$~be a limit ordinal and $(E_i\to E_j)_{0\le i<j<\alpha}$ be
a well-ordered chain of morphisms in $\sA$ with the objects $E_i\in\sC$
for all $0\le i<\alpha$.
 For every successor ordinal $\beta=\gamma+1<\alpha$, we have
$\varinjlim_{i<\beta}E_i=E_\gamma$.
 The canonical presentation of this directed colimit is a split short
exact sequence
\begin{equation} \label{split-direct-limit-presentation}
 0\lrarrow K_\gamma=\bigoplus\nolimits_{i<\gamma}E_i\lrarrow
 \bigoplus\nolimits_{i\le\gamma}E_i\overset p\lrarrow E_\gamma\lrarrow0.
\end{equation}
 Here the components of the split epimorphism~$p$ are the transition
morphisms $E_i\rarrow E_\gamma$, \,$0\le i\le\gamma$.
 One can easily see that the morphism~$p$ is naturally isomorphic to
the direct summand projection $\bigoplus\nolimits_{i\le\gamma}E_i
\rarrow E_\gamma$, so the kernel $K_\gamma$ of~$p$ is naturally
identified with $\bigoplus_{i<\gamma}E_i$.

 As the ordinal~$0\le\gamma<\alpha$ varies, the short exact
sequences~\eqref{split-direct-limit-presentation} form a directed
diagram whose directed colimit is the canonical presentation
\begin{equation} \label{nonsplit-direct-limit-presentation}
 0\lrarrow K_\alpha\lrarrow
 \bigoplus\nolimits_{i<\alpha}E_i\lrarrow E_\alpha\lrarrow0
\end{equation}
of the directed colimit $E_\alpha=\varinjlim_{i<\alpha}E_i$.
 The sequence~\eqref{nonsplit-direct-limit-presentation} is exact
in~$\sA$ as the directed colimit of short exact
sequences~\eqref{split-direct-limit-presentation}.
 So we have $K_\alpha=\varinjlim_{0\le\gamma<\alpha}K_\gamma$.

 The class $\sC\subset\sA$ is closed under finite direct sums by
assumption; hence it is clear from the condition~(3) that it is closed
under set-indexed coproducts.
 Therefore, we have $K_\gamma\in\sC$ for all $0\le\gamma<\alpha$.
 Furthermore, the transition morphisms in the directed diagram
$(K_\gamma\to K_\delta)_{0\le\gamma<\delta<\alpha}$
are split monomorphisms (because $K_\gamma$ is a direct summand
in $\bigoplus_{i\le\gamma}E_i$ and the transition morphism
$\bigoplus_{i\le\gamma}E_i\rarrow\bigoplus_{i\le\delta}E_i$ is
the subcoproduct injection, i.~e., the standard split monomorphism).

 Thus the property~(3) can be applied to the effect that
$K_\alpha\in\sC$.
 Finally, we conclude that $E_\alpha\in\sC$, since the class $\sC$ is
closed under the cokernels of admissible monomorphisms in $\sA$
by assumption.

 (4)\,$\Longrightarrow$\,(5) is a general property of directed
colimits; see~\cite[Sections~1.5\+-1.7]{AR}.
\end{proof}

\Section{Cotorsion Periodicity in Category-Theoretic Context}
\label{categorical-cotorsion-periodicity-secn}

 Let $\sK$ be an exact category.
 Following~\cite[Section~7.1]{Pedg}, a full subcategory $\sA\subset\sK$
is said to be \emph{self-resolving} if it is self-generating (as defined
in Section~\ref{prelim-ext-orthogonal-secn}), closed under extensions,
and closed under the kernels of admissible epimorphisms.

 In particular, $\sA\subset\sK$ is said to be
\emph{resolving}~\cite[Section~2]{Sto} if it is generating and closed
under extensions and the kernels of admissible epimorphisms.
 Clearly, any resolving full subcategory is self-resolving.

 The following generalization of~\cite[Theorem~4.7]{BCE} is our
main category-theoretic cotorsion periodicity theorem.

\begin{thm} \label{categorical-periodicity-theorem}
 Let\/ $\sK$ be an exact category and\/ $\sA\subset\sK$ be
a self-resolving subcategory.
 Assume that\/ $\sA$ (with its inherited exact category structure) is
an exact category with exact directed colimits, and the inclusion
functor\/ $\sA\rarrow\sK$ preserves directed colimits.
 Put\/ $\sB=\sA^{\perp_1}=\sA^{\perp_{\ge1}}\subset\sK$ (as per
Lemma~\ref{garcia-rozas}), and let $M\in\sK$ be a\/ $\sB$\+periodic
object.
 Then the class of objects $\sA\cap{}^{\perp_1}M\subset\sK$ is closed
in\/~$\sA$ under
\begin{itemize}
\item extensions,
\item the cokernels of admissible monomorphisms,
\item the kernels of admissible epimorphisms and
\item directed colimits.
\end{itemize}
\end{thm}

\begin{proof}
 Put $\sC_M=\sA\cap{}^{\perp_1}M$.
 Then, by Lemma~\ref{periodic-2-out-of-3-lemma}, the class of objects
$\sC_M\subset\sA$ is closed under extensions, the cokernels of
admissible monomorphisms, and the kernels of admissible epimorphisms 
in~$\sA$.
 By Lemma~\ref{eklof-lemma}, the class $\sC_M$ is also closed under
transfinitely iterated extensions in~$\sA$ (notice that all directed
colimits, and consequently all transfinitely iterated extensions
in $\sA$ remain such in $\sK$ by assumption).
 Applying Proposition~\ref{transfinite-extensions-direct-limits}, we
conclude that the class $\sC_M$ is closed under directed colimits
in~$\sA$.
\end{proof}

 The following corollary shows how
Theorem~\ref{categorical-periodicity-theorem} can be applied.

\begin{cor} \label{categorical-periodicity-cor}
 Let\/ $\sK$ be an exact category and\/ $\sA\subset\sK$ be
a self-resolving subcategory.
 Assume that\/ $\sA$ (with its inherited exact category structure) is
an exact category with exact directed colimits, and the inclusion
functor\/ $\sA\rarrow\sK$ preserves directed colimits.
 Assume further that the smallest subcategory\/ $\sA'\subset\sA$ 
such that
\begin{itemize} 
\item $\sA'$ contains all the objects of\/ $\sA$ that have finite
projective dimension in\/~$\sK$,
\item $\sA'$ is closed under directed colimits in\/~$\sA$,
\item $\sA'$ is closed under extensions in\/~$\sA$,
\item $\sA'$ is closed under the cokernels of admissible monomorphisms
in\/~$\sA$ and
\item $\sA'$ is closed under the kernels of admissible epimorphisms
in\/~$\sA$
\end{itemize}
coincides with $\sA$ itself.
 Put\/ $\sB=\sA^{\perp_1}=\sA^{\perp_{\ge1}}\subset\sK$.
 Then any\/ $\sB$\+periodic object in\/ $\sK$ belongs to\/~$\sB$.
\end{cor}

\begin{proof}
 Let $M\in\sK$ be a $\sB$\+periodic object in $\sK$ and put
$\sC_M=\sA\cap{}^{\perp_1}M$.
 Let $\sP$ denote the intersection of $\sA$ with the full subcategory
of all objects of finite projective dimension in~$\sK$.
 By Lemma~\ref{fin-proj-dim-are-ok}, we have $\sP\subset\sC_M$.
 By Theorem~\ref{categorical-periodicity-theorem}, the class $\sC_M$ is
closed under extensions, directed colimits, the cokernels of admissible
monomorphisms, and the kernels of admissible epimorphisms in~$\sA$.
 By assumption $\sC_M=\sA$, and so $M\in\sA^{\perp_1}=\sB$.
\end{proof}

 In the rest of this section, we discuss a category-theoretic version
of a standard technique~\cite[proof of Proposition~7.6]{CH},
\cite[Propositions~1 and~2]{EFI} allowing to apply periodicity theorems
to the study of the objects of cocycles in acyclic complexes.
 The point is that we consider classes closed under infinite products
\emph{without} assuming the infinite products to be exact in our exact
category.
 This makes the argument more complicated with additional assumptions
required.

\begin{lem} \label{product-of-exact-sequences}
 Let\/ $0\rarrow K_i\rarrow L_i\rarrow M_i\rarrow0$ be a family of
admissible short exact sequences in an exact category\/~$\sK$.
 Assume that the products\/ $\prod_iK_i$, \ $\prod_iL_i$, \
$\prod_i M_i$ exist in\/~$\sK$.
 Assume further that there exists an admissible epimorphism
$q\:A\rarrow\prod_i M_i$, where $A\in\sK$ is an object such that\/
$\Ext^1_\sK(A,K_i)=0$ for all indices~$i$.
 Then
$$
 0\lrarrow \prod\nolimits_iK_i \overset k\lrarrow \prod\nolimits_iL_i
 \overset p\lrarrow \prod\nolimits_i M_i \lrarrow0
$$
is an admissible short exact sequence in\/~$\sK$.
\end{lem}

\begin{proof}
 It is clear that $k=\ker(p)$, as products commute with kernels in
any category.
 Therefore, it suffices to prove that $p$~is an admissible epimorphism
in~$\sK$.
 For this purpose, we show that the admissible epimorphism~$q$
factorizes through~$p$; then the ``obscure axiom''
(the dual version of~\cite[Proposition~2.16]{Bueh}) applies.

 Indeed, in order to show that the morphism $q\:A\rarrow\prod_iM_i$
factorizes through the morphism $p\:\prod_iL_i\rarrow\prod_i M_i$,
it suffices to check that the composition $A\rarrow\prod_iM_i
\rarrow M_j$ factorizes through the admissible epimorphism
$L_j\rarrow M_j$ for every index~$j$.
 The latter follows from the assumption that $\Ext^1_\sK(A,K_j)=0$.
\end{proof}

 The following proposition is our category-theoretic
version of~\cite[Proposition~2]{EFI} not assuming exactness of
countable products.

\begin{prop} \label{periodicity-and-cocycles}
 Let\/ $\sK$ be an idempotent-complete exact category with (possibly
nonexact) countable products, and let\/ $\sB\subset\sK$ be a full
subcategory closed under direct summands and countable products.
 Assume that every object of\/ $\sK$ is an admissible quotient of
an object of finite projective dimension in\/ $\sK$ belonging
to\/~${}^{\perp_{\ge1}}\sB$.
 Then the following conditions are equivalent:
\begin{enumerate}
\item all\/ $\sB$\+periodic objects of\/ $\sK$ belong to\/~$\sB$;
\item in any (unbounded) acyclic complex in\/ $\sK$ with the terms
in\/ $\sB$, the objects of cocycles belong to\/~$\sB$.
\end{enumerate}
\end{prop}

\begin{proof}
 (2)\,$\Longrightarrow$\,(1) Given a $\sB$\+periodic object $M\in\sK$,
produce an unbounded acyclic complex in $\sK$ with the terms in $\sB$
by splicing copies of the short exact sequence $0\rarrow M\rarrow B
\rarrow M\rarrow0$, \ $B\in\sB$ infinitely in both directions.
Then $M\in\sB$ by condition~(2).

 (1)\,$\Longrightarrow$\,(2) This is the nontrivial implication
requiring our additional assumption.
 Basically, the argument consists in chopping up a given acyclic
complex into admissible short exact sequences and taking the product of
the resulting pieces.

 Indeed, by definition, any acyclic complex $B^\bu$ in\/ $\sK$ is
produced by splicing admissible short exact sequences $0\rarrow M^i
\rarrow B^i\rarrow M^{i+1} \rarrow0$, \ $i\in\boZ$.
 In the situation at hand, we have $B^i\in\sB$ and $M^i\in\sK$, and
need to show that $M^i\in\sB$.
 Taking the product of these short exact sequences over $i\in\boZ$,
we obtain a sequence
\begin{equation} \label{product-sequence}
 0\lrarrow\prod\nolimits_{i\in\boZ} M^i \lrarrow
 \prod\nolimits_{i\in\boZ} B^i \lrarrow
 \prod\nolimits_{i\in\boZ} M^i \lrarrow0.
\end{equation}

 In order to show that~\eqref{product-sequence} is admissible exact,
we apply Lemma~\ref{product-of-exact-sequences}.
 By assumption, there exists an object $A\in\sK$ of finite projective
dimension, belonging to ${}^{\perp_{\ge1}}\sB\subset\sK$, and
an admissible epimorphism $A\rarrow\prod_{i\in\boZ}M^i$ in~$\sK$.
 It remains to check that $\Ext^1_\sK(A,M^i)=0$ for all $i\in\boZ$.
 This is a version of Lemma~\ref{fin-proj-dim-are-ok}, provable by
a similar dimension shifting: one constructs isomorphisms
$\Ext^1_\sK(A,M^i)\simeq\Ext^2_\sK(A,M^{i-1})\simeq
\Ext^3_\sK(A,M^{i-2})\simeq\dotsb$ and uses the finiteness of
projective dimension of~$A$.

 Now we have $\prod_{i\in\boZ}B^i\in\sB$, since $\sB$ is closed under
countable products in~$\sK$.
 By~(1), we can conclude that $\prod_{i\in\boZ}M^i\in\sB$.
 Since $\sB$ is closed under direct summands in $\sK$, it follows
that $M^i\in\sB$ for every~$i$.
\end{proof}

\Section{Cotorsion Periodicity for Quasi-Coherent Sheaves}
\label{qcoh-cotorsion-periodicity-secn}

 Having Corollaries~\ref{direct-limits-of-flat-qcoh-of-finite-projdim}
and~\ref{categorical-periodicity-cor} at our disposal, the proof
of cotorsion periodicity for quasi-coherent sheaves on quasi-compact
semi-separated schemes is now straightforward.
 Let us spell it out.

 Let $X$ be a quasi-compact semi-separated scheme.
 A quasi-coherent sheaf $\C$ on $X$ is called \emph{cotorsion} if
$\Ext^1_X(\F,\C)=0$ for all flat quasi-coherent sheaves on~$X$.
 We denote the class of all flat quasi-coherent sheaves by
$X\qcoh_\fl\subset X\qcoh$ and the class of all cotorsion quasi-coherent
sheaves by $X\qcoh^\cot\subset X\qcoh$.

\begin{lem} \label{flats-hereditary-lemma}
 Let $X$ be a quasi-compact semi-separated scheme.
 Then the class of all flat quasi-coherent sheaves $X\qcoh_\fl$
is resolving and closed under directed colimits in $X\qcoh$.
 Consequently, one has $\Ext^n_X(\F,\C)=0$ for all flat quasi-coherent
sheaves $\F$, all cotorsion quasi-coherent sheaves $\C$, and all
integers $n\ge1$.
\end{lem}

\begin{proof}
 In fact, the pair of classes $(X\qcoh_\fl$, $X\qcoh^\cot)$ is
a hereditary complete cotorsion pair in $X\qcoh$; but we do not even
need the definitions of these terms.
 The class $X\qcoh_\fl$ is generating in $X\qcoh$
by~\cite[Section~2.4]{M-n} or~\cite[Lemma~A.1]{EP}.
 The fact that this class is closed under extensions, kernels of
epimorphisms, and directed colimits follows from the local nature of
the definition of a flat quasi-coherent sheaf and the similar
properties of the class of flat modules over a ring.
 This proves the first assertion of the lemma; the second one is
then provided by Lemma~\ref{garcia-rozas}.
\end{proof}

 The following cotorsion periodicity theorem is the quasi-coherent sheaf
version of the cotorsion periodicity theorem for modules over
associative rings~\cite[Theorem~1.2(2), Proposition~4.8(2), or
Theorem~5.1(2)]{BCE}.

\begin{thm} \label{qcoh-cotorsion-periodicity-theorem}
 Let $X$ be a quasi-compact semi-separated scheme.
 Then any cotorsion-periodic quasi-coherent sheaf on $X$ is cotorsion.
\end{thm}

\begin{proof}
 Directed colimits are exact in $\sK=X\qcoh$, so it follows from
Lemma~\ref{flats-hereditary-lemma} that they are also exact in
the exact category $\sA=X\qcoh_\fl$ of flat quasi-coherent sheaves.
 The assertion of the theorem is now obtained by applying
Corollary~\ref{categorical-periodicity-cor}, whose assumptions are
satisfied by
Corollary~\ref{direct-limits-of-flat-qcoh-of-finite-projdim}
and Lemma~\ref{flats-hereditary-lemma}.
\end{proof}

\begin{rem}
 Notice that the proof of
Theorem~\ref{qcoh-cotorsion-periodicity-theorem} above does not use
the extension, kernel, and cokernel clauses of
Corollary~\ref{categorical-periodicity-cor}, but only the directed
colimit clause.
 An alternative approach to proving the quasi-coherent cotorsion
periodicity, using the kernel closure and the \v Cech coresolution,
was previously suggested in the paper~\cite[Theorem~3.3]{CET}.
 Let us point out that~\cite[first paragraph of the proof of
Lemma~3.2]{CET} is erroneous, but the assertion of~\cite[Lemma~3.2]{CET}
is valid despite of the mistake, which was subsequently corrected
in~\cite[Theorem~6.3(2)]{EGilO}; see also~\cite[Remark~6.7]{EGilO}.

 The argument in~\cite{CET} works directly with acyclic complexes
with cotorsion components, but one can adapt the idea to our setting
as follows.
 The \v Cech coresolution~\eqref{cech-coresolution-eqn} from
Lemma~\ref{cech-coresolution} shows that any flat quasi-coherent
sheaf on $X$ can be obtained from the direct images of flat
quasi-coherent sheaves with respect to open immersions of affine
open subschemes $j\:U\rarrow X$ using the operations of finite
direct sum and kernel of admissible epimorphism in $X\qcoh_\fl$.
 One observes that, for any flat quasi-coherent sheaf $\G$ on $U$,
the direct image $j_*\G$ is easily constructed as a directed colimit
of flat quasi-coherent sheaves of finite projective dimension on~$X$
(using the Govorov--Lazard theorem for $\O(U)$\+modules and
Theorem~\ref{qcoh-projdim-theorem}).
 So our Corollary~\ref{categorical-periodicity-cor}, with its directed
colimit and kernel clauses, is applicable.
\end{rem}

\begin{cor} \label{qcoh-cotorsion-periodicity-stated-for-complexes}
 Let $X$ be a quasi-compact semi-separated scheme and $\B^\bu$
be an acyclic complex in $X\qcoh$ whose terms $\B^n$ are cotorsion
quasi-coherent sheaves.
 Then the sheaves of cocycles of the complex $\B^\bu$ are also
cotorsion.
\end{cor}

\begin{proof}
 We apply Proposition~\ref{periodicity-and-cocycles} for
$\sK=X\qcoh$ and $\sB=X\qcoh^\cot$.
 By Theorem~\ref{qcoh-cotorsion-periodicity-theorem}, condition~(1)
of the proposition is satisfied; we want to deduce condition~(2).
 It remains to check the assumptions of the proposition.

 For any scheme $X$, the category $X\qcoh$ is Grothendieck, so it has
infinite products.
 The full subcategory of cotorsion sheaves $X\qcoh^\cot$ is closed under
infinite products in $X\qcoh$ by~\cite[Corollary A.2]{CoSt}
or the dual version of~\cite[Corollary 8.3]{CoFu}; it is also obviously
closed under direct summands.
 Finally, any quasi-coherent sheaf on $X$ is a quotient sheaf of
a flat quasi-coherent sheaf locally of projective dimension~$\le1$ by
Remark~\ref{enough-countably-flats-remark}.
 Such quasi-coherent sheaves have finite projective dimension in
$X\qcoh$ by Theorem~\ref{qcoh-projdim-theorem}.
 One has $X\qcoh_\fl\subset{}^{\perp_{\ge1}}(X\qcoh^\cot)$
by Lemma~\ref{flats-hereditary-lemma}.
\end{proof}

 To end, let us formulate our intended application of
Theorem~\ref{qcoh-cotorsion-periodicity-theorem}, viz.,
a description of the derived category of quasi-coherent sheaves
$\sD(X\qcoh)$ in terms of cotorsion quasi-coherent sheaves.

 Given an additive category $\sE$, we denote by $\Hot(\sE)$
the homotopy category of (unbounded) complexes in~$\sE$.
 Given an exact category $\sK$, we denote by $\sD(\sK)$ the (unbounded)
derived category of~$\sK$.
 So $\sD(\sK)$ is the triangulated Verdier quotient category
$\sD(\sK)=\Hot(\sK)/\Ac(\sK)$, where $\Ac(\sK)\subset\Hot(\sK)$ is
the triangulated subcategory of acyclic complexes.

\begin{lem} \label{subcategory-localization-lemma}
 Let\/ $\sK$ be an idempotent-complete exact category and\/
$\sB\subset\sK$ be a full additive subcategory.
 Assume that for any complex $K^\bu$ in\/ $\sK$ there exists a complex
$B^\bu$ in\/ $\sB$ together with a morphism of complexes $K^\bu\rarrow
B^\bu$ which is a quasi-isomorphism of complexes in\/~$\sK$.
 Then the inclusion of additive categories $\sB\rarrow\sK$ induces
a triangulated equivalence of Verdier quotient categories
$$
 \frac{\Hot(\sB)}{\Hot(\sB)\cap\Ac(\sK)}\overset\simeq\lrarrow
 \frac{\Hot(\sK)}{\Ac(\sK)}=\sD(\sK).
$$
\end{lem}

\begin{proof}
 This is a particular case of \cite[Corollary~7.2.2 or
Proposition~10.2.7(i)]{KS} or~\cite[Lemma~1.6(b)]{Pkoszul}.
\end{proof}

\begin{prop} \label{grothendieck-category-localization}
 Let\/ $\sK$ be a Grothendieck abelian category and\/ $\sB\subset\sK$
be a full additive subcategory containing all the injective objects
of\/~$\sK$.
 Then the inclusion of additive categories $\sB\rarrow\sK$ induces
a triangulated equivalence of Verdier quotient categories
$$
 \frac{\Hot(\sB)}{\Hot(\sB)\cap\Ac(\sK)}
 \overset\simeq\lrarrow\sD(\sK).
$$
\end{prop}

\begin{proof}
 The point is that the assumption of
Lemma~\ref{subcategory-localization-lemma} can be satisfied by
choosing $B^\bu$ to be a suitable complex of injective objects in~$\sK$.
 There are even several ways to do it: e.~g., one can choose $B^\bu$
to be a homotopy injective complex of injective objects, as there are
enough such complexes in any Grothendieck
category~\cite[Theorem~3.13 and Lemma~3.7(ii)]{Ser},
\cite[Corollary~7.1]{Gil}, \cite[Corollary~8.5]{PS4}.
 Alternatively, choosing $B^\bu$ as an arbitrary complex of injectives,
one can make the cone of the morphism $K^\bu\rarrow B^\bu$ not just
an acyclic, but a \emph{coacyclic complex in the sense of Becker},
which is a stronger property~\cite[Corollary~9.5]{PS4}.
\end{proof}

 For specific examples of categories $\sK$ and $\sB$, there are likely 
many further alternative options of choosing a quasi-isomorphism
appearing in the proof of
Proposition~\ref{grothendieck-category-localization}.
 For example, in the case $\sK=X\qcoh$ and $\sB=X\qcoh^\cot$ considered
in the corollary below, one can choose $B^\bu$ as an arbitrary complex
of cotorsion sheaves and make $K^\bu\rarrow B^\bu$ a termwise
monomorphism whose cokernel is an acyclic complex of flat sheaves with
flat sheaves of cocycles.
 This is a quasi-coherent version of~\cite[Theorem~5.3]{BCE} based on
Corollary~\ref{qcoh-cotorsion-periodicity-stated-for-complexes} above.

\begin{cor}
 Let $X$ be a quasi-compact semi-separated scheme.
 Then the inclusion of exact/abelian categories $X\qcoh^\cot\rarrow
X\qcoh$ induces an equivalence of their unbounded derived categories,
$$
 \sD(X\qcoh^\cot) \overset\simeq\lrarrow \sD(X\qcoh).
$$
\end{cor}

\begin{proof}
 Compare Proposition~\ref{grothendieck-category-localization} with
Corollary~\ref{qcoh-cotorsion-periodicity-stated-for-complexes}.
\end{proof}

\end{document}